\documentclass[12pt,a4paper]{article}
\usepackage{fancyhdr}
\usepackage[bottom=1.5in, top=1.2in]{geometry}

\usepackage{amssymb,amsmath,amsthm}
\usepackage[english]{babel}
\usepackage{t1enc}
\usepackage[utf8]{inputenc}
\usepackage{epsfig}
\usepackage{subfig}
\usepackage{epstopdf}
\usepackage{xcolor}
\usepackage{comment}
\usepackage{hyperref}
\usepackage[shortlabels]{enumitem}

\newtheorem{thm}{Theorem}
\newtheorem{cor}[thm]{Corollary}

\newtheorem{lem}[thm]{Lemma}
\newtheorem{prop}[thm]{Proposition}

\newtheorem{obs}[thm]{Observation}

\theoremstyle{definition}
\newtheorem{defi}[thm]{Definition}
\newtheorem{rem}[thm]{Remark}

\usepackage{indentfirst}
\def\\deg{\mathrm {\deg}}
\def\wt{\mathrm {wt}}

\def\red1{\color{red}1\color{black}}
\def\blue1{\color{blue}1\color{black}}

\def\01{0\text{-}1}

\makeatletter
\def\blfootnote{\gdef\@thefnmark{}\@footnotetext}
\makeatother

\title{Newton polygons for the non-bipartite dimer model}
\author{Vladimir Bo\v skovi\'c \thanks{Universit\'e Paris-Saclay, CNRS, CEA, Institut de Physique Th\'eorique, 91191 Gif-sur-Yvette, France\\Email address: vladimir.boskovic at ipht dot fr}}

\begin{document}

\maketitle

\begin{abstract}
    We study the dimer model on two families of non-bipartite graphs on the torus. The first family is obtained by replacing degree $3$ vertices in a bipartite torus graph with triangles, while the second consists of corner graphs associated with bipartite torus graphs. We determine the relationship between the Newton polygons of these graphs and those of the underlying bipartite graphs. We also identify the primitive edge vectors of the Newton polygons with the homology classes of the zig-zag paths.

    We further consider the marginal polynomials obtained by restricting to monomials corresponding to a boundary side of the Newton polygon. For the triangular lattice and the Fisher graph of the hexagonal lattice, we prove that these polynomials are real-rooted and obtain an explicit factorization of their roots. Finally, we introduce new local moves, including a move on non-planar graphs, that preserve the dimer partition functions up to a scale.
\end{abstract}

\section{Introduction}

The dimer model is a statistical mechanics model that studies random perfect matchings on graphs. Kasteleyn~\cite{Kasteleyn63}, and Temperley and Fisher~\cite{TemperleyFisher61} in the 1960s described the partition functions of this model for planar graphs as a determinant of a signed weighted adjacency matrix, which is usually referred to in the literature as the Kasteleyn matrix. Over the past three decades, several mathematical aspects of the dimer model have been extensively studied. In particular, the study of limit shapes was initiated~\cite{CKP01}, and deep connections to real algebraic geometry~\cite{KO06, KOS06} and integrable systems~\cite{GoncharovKenyon} were found. 

These fundamental works have initiated several different lines of research that have developed at a strong pace recently. A common assumption in these results is that the underlying graph structure is bipartite, namely that every cycle has an even length. For an introduction to the dimer model, see the lecture notes by Kenyon~\cite{KenyonLectureNotes} and by de Tilière~\cite{deTiliere2015}. 

In this paper, we investigate the combinatorial and topological properties of dimer models on two broad classes of non-bipartite graphs. We study infinite planar graphs that are periodic in the horizontal and vertical directions. The dimer model on these graphs can be studied by projecting them down to a torus obtained by identifying the opposite sides of a finite fundamental domain. Then, to this finite torus graph, we can associate a Laurent polynomial, which is called the characteristic polynomial. When a graph is bipartite, the coefficients of this polynomial describe the dimer partition function with a prescribed homology class. As we will see later, the characteristic polynomial in the non-bipartite case carries different combinatorial information.

Kenyon, Okounkov, and Sheffield~\cite{KOS06} used Newton polygons to parametrize the ergodic Gibbs measures and encode the phase diagram of the bipartite dimer model. The interior (resp. boundary) integer points describe smooth (resp. frozen) regions of the model, while the interior non-integer points correspond to rough regions. Later on, this became a very important tool for investigating the limit shapes of many planar bipartite graphs.

In their works~\cite{KO06,KOS06}, they characterized the spectral curves of the bipartite dimer model, that is, the zeros of the characteristic polynomial. They showed that these curves are, in fact, simple Harnack curves. On the other hand, they also showed that any such curve can be obtained as the spectral curve of some bipartite dimer model. Several characterizations of such curves are given by Mikhalkin~\cite{Mikhalkin00}. As described in~\cite{MikhalkinRullgard}, simple Harnack curves are the curves that maximize the area of their amoeba among all curves with the same Newton polygon. 

Understanding the spectral curves for the non-bipartite dimer model has been a widely open question. Cimasoni and Duminil-Copin~\cite{CimasoniDuminilCopin}, and Li~\cite{ZLi} showed that when the non-bipartite graphs are taken to be Fisher graphs, with weights related to the ferromagnetic Ising model, their spectral curves remain simple Harnack. On the other hand, the curves derived from the anti-ferromagnetic Ising model appear to have non-Harnack properties. These curves associated with isoradial graphs were recently studied by de Tilière and Rey~\cite{deTiliereRey} and classified in the genus one case.

Boutillier and de Tilière~\cite{BoutillierdeTiliere2010} showed that the Newton polygons of the dimer characteristic polynomials related to the critical Ising model on isoradial graphs are equal to the Newton polygon of the Laplacian characteristic polynomials. In their work, they have already observed certain properties of the non-bipartite dimer characteristic polynomial, which will be useful throughout this paper.

The Newton polygons played another important role in the work of Goncharov and Kenyon~\cite{GoncharovKenyon}, in which they introduced a discrete integrable system associated with a Newton polygon. In their construction, they identified a family of minimal bipartite torus graphs whose characteristic polynomials have the same Newton polygon, and they classified the local moves that relate such graphs. More recently, Galashin and George~\cite{GalashinGeorge} extended their result to a more general family of move-reduced graphs, which have the same decorated Newton polygon, where the decoration arises from the partitions of the boundary sides of the polygon.

To formulate our results, we now introduce some necessary notions. A \emph{zig-zag path} on a torus graph $\Gamma$ is defined as an oriented path that turns alternately maximally left and maximally right. A graph is \emph{isoradial} if, when lifted to its universal cover, there are no self-intersecting zig-zag paths and no two zig-zag paths intersect more than once (that is, they can share at most one edge). Since zig-zag paths are on the torus graph, they describe a closed loop, and one can associate a homology class with such loops. For more details, see Section~\ref{trunc_graphs}. 

Let $P(z,w)$ be a Laurent polynomial. Its \emph{Newton polygon} denoted by $N(P)$, is defined as the convex hull of integer points $(i,j) \in \mathbb{Z}^2$ such that $z^i w^j$ is a non-zero monomial in $P(z,w)$. If the polynomial $P(z,w)$ is the \emph{dimer characteristic polynomial} of a torus graph $\Gamma$, then we denote the Newton polygon by $N(\Gamma)$. See Section \ref{Kast_matrix} for a precise definition of the characteristic polynomial.

Let $V_3 = \{v \in V(\Gamma): \deg(v) = 3\}$ and $S$ be a subset of $V_3$. Let $\Gamma_{\Delta_{S}}$ be the graph obtained by replacing all the vertices of $S$ with triangles. If $S = V(\Gamma)$, then we simply write $\Gamma_\Delta$. See Definition \ref{const of gamma delta} for a more precise definition. See also Figure \ref{fig:hex_fish_cor} for an example. Now, we can state the first of our main results.

\begin{thm}
    Let $\Gamma$ be an isoradial bipartite torus graph. Then the Newton polygons of $\Gamma$ and $\Gamma_{\Delta_{S}}$ are the same for any $S \subseteq V_3$. Furthermore, if $\Gamma$ is a $3$-valent graph, then the primitive edge vectors of the Newton polygon $N(\Gamma_\Delta)$ are obtained from the homology classes of zig-zag paths of $\Gamma_{\Delta}$.
\end{thm}

We discuss its proof and related results in Section \ref{trunc_graphs}.

\begin{figure}[h!]
    \centering
    \includegraphics[width=0.95\textwidth]{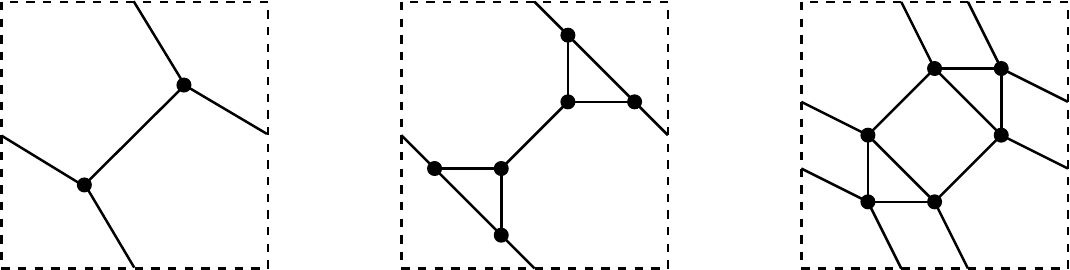}
    \caption{\label{fig:hex_fish_cor} From left to right: the hexagonal lattice $H$, the hexagonal Fisher graph $H_\Delta$, and the hexagonal corner graph $C_H$.}
\end{figure}

The next result goes in the same direction, but for a different family of graphs. We define \emph{corner graphs} by placing a \emph{corner} for every pair of a vertex and a face adjacent to it. Then we connect two corners if they are neighbors around a vertex or a face. The final graph is a $4$-valent graph, which is non-bipartite if it contains a vertex of odd degree. These graphs appeared in the work of Chelkak~\cite{Chelkak} in the context of the Ising model. See Figure \ref{fig:hex_fish_cor} for an example of the corner graph applied to the hexagonal lattice. This example was previously considered in~\cite{KenyonSunWilson}, where a certain property of its spectral curve was shown. 

\begin{thm} \label{mainthm2}
    Let $\Gamma$ be an isoradial bipartite torus graph, and let $C_\Gamma$ be its corner graph. Then $N(C_\Gamma) = 2 N(\Gamma)$ and the primitive edge vectors of $N(C_\Gamma)$ are obtained from the homology classes of its zig-zag paths.
\end{thm}

The proof of Theorem \ref{mainthm2}, as well as other results about corner graphs, can be found in Section \ref{corner_graphs}.

In Section \ref{marginal_polys}, we study the \emph{marginal polynomials} of several families of graphs with an arbitrarily sized fundamental domain. These polynomials are defined as one-variable polynomials obtained by restricting the characteristic polynomial to the monomials corresponding to one boundary side of the Newton polygon. For bipartite graphs, the marginal polynomials have real roots, which is a consequence of the result that the spectral curves for the bipartite dimer model with positive real edge weights are simple Harnack~\cite{KO06,KOS06}. For non-bipartite graphs, this question has remained open. Here, we answer it for a large family of graphs.

\begin{thm}
    Let $\Gamma$ be an isoradial bipartite torus graph, and let $S \subseteq V(\Gamma)$ be a subset of vertices of degree $3$. Then the marginal polynomials of $N(\Gamma_{\Delta_S})$ attached to the characteristic polynomial are real-rooted.

    Moreover, let $N$ be the Newton polygon associated with the dimer model on a fundamental domain of arbitrary size of the triangular lattice or the Fisher graph of the hexagonal lattice. Then, the marginal polynomials of the characteristic polynomial attached to $N$ are real-rooted, and there is an explicit formula for their roots in factorized form in terms of the zig-zag paths with prescribed homology.
\end{thm}

Section \ref{local_moves} introduces a somewhat different question, not necessarily restricted to torus graphs and the notion of Newton polygons. We are interested in studying the local moves (or transformations) on graphs that preserve the dimer partition functions up to a scale. We introduce two new local moves on non-bipartite planar graphs that respect this property. They can be found in Figures \ref{fig:localmove1} and \ref{fig:nbspider}. 

We also discuss the relation between the moves that preserve the dimer partition function and the moves that preserve the characteristic polynomials. These consequently generalize the main results of the paper by enlarging the class of graphs to which these results apply. We finish the section by introducing yet another move, a \emph{cross move}, which, under certain conditions, forces a graph to be non-planar.

\section{Preliminaries}

The goal of this section is to introduce several notions that will be used throughout the paper. These include the Kasteleyn matrix and the modified Kasteleyn matrix, which are used to compute the dimer partition function for planar graphs and the dimer characteristic polynomial for torus graphs, respectively. 

These notions have a slightly simpler definition when the graph is bipartite. In that case, the bipartite Kasteleyn matrix is essentially the top right block of the general Kasteleyn matrix, viewed as a $2 \times 2$ block matrix. All the information of the general Kasteleyn matrix is contained in that top-right block when the graph is bipartite, so they are basically equivalent.

We also recall several important combinatorial properties that are carried by the characteristic polynomials. Finally, we introduce the notions of the Newton polygon and zig-zag paths and describe how they are related.

\subsection{Kasteleyn matrix} \label{Kast_matrix}

A \emph{dimer cover}, or a perfect matching, is a subset of edges of a graph such that each vertex is incident to a unique edge. Let $G = (V,E)$ be a finite graph. Let $M(G)$ be the set of all dimer covers in $G$. We assign positive weights to edges via a function $\wt: E \rightarrow \mathbb{R}^{+}$. For a given dimer cover $M \in M(G)$, the weight of $M$ is obtained as the product \[\wt(M) = \prod_{e \in M} \wt(e).\] 
The \emph{dimer partition function} is the weighted sum over all dimer covers in $G$,

\[Z(G) = \sum_{M \in M(G)} \wt(M).\]

We further define a \emph{Boltzmann probability measure} for this model such that, for a dimer cover $M \in M(G)$, we have 

\[\mathbb{P}(M) = \frac{\prod_{e \in M} \wt(e)}{Z(G)}.\]

By multiplying every edge incident to a given vertex by $\lambda > 0$, the probability measure does not change, as every weight of a dimer cover is multiplied by $\lambda$, as well as the partition function. These operations are called \emph{gauge transformations}. 

A fundamental tool for studying the planar dimer model is the \emph{Kasteleyn matrix}. Suppose from now on that $G$ is a planar graph. We orient the edges of $G$ while respecting the so-called \emph{clockwise-odd rule}: for each inner face of $G$, there is an odd number of edges that are oriented clockwise. 

Let $\mathcal{E}(u,v)$ be the set of edges between the vertices $u$ and $v$. Using the chosen orientation of the graph $G$, we introduce the \emph{Kasteleyn sign} of an edge $e \in \mathcal{E}(u,v)$, denoted by $\kappa(e)$, which is equal to $1$ (resp. $-1$) if the edge $e$ is oriented from $u$ to $v$ (resp. from $v$ to $u$).  We define the Kasteleyn matrix $K$ as a signed adjacency matrix of $G$. For all $u,v \in V$:

\begin{align*}
K_{u,v} = \sum_{e \in \mathcal{E}(u,v)}\kappa(e) \wt(e).
\end{align*}

The partition functions for the dimer model were computed by Kasteleyn~\cite{Kasteleyn63}, and Temperley and Fisher~\cite{TemperleyFisher61} in terms of the Kasteleyn matrix
\[Z(G) = \sqrt{\det K}.\]

\subsubsection*{Torus graphs}

We now generalize the notion of the Kasteleyn matrix to graphs on a torus. A graph $\Gamma = (V, E)$ drawn on the torus $\mathbb{T} = \mathbb{R}^2/\mathbb{Z}^2$ is a \emph{torus graph} if every connected component of the complement of $\Gamma$ in the torus is a topological disk. Let $R$ be a fundamental domain of rectangular shape. From now on, we assume that $\Gamma$ always has a perfect matching, multiple edges are allowed, but self-loops are not.

\begin{figure}[h!]
    \centering
    \includegraphics[width=0.75\textwidth]{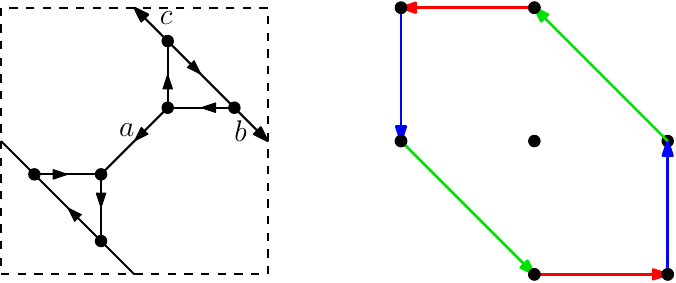}
    \caption{\label{fig:Fisher1x1} Fisher graph $F_{1,1}$ with its Kasteleyn orientation, positive real edge weights $a,b,c$ and the remaining edges have weight $1$, and the Newton polygon for its characteristic polynomial.}
\end{figure}

We modify the definition of the Kasteleyn matrix to make it depend on the variables $z, w \in \mathbb{C}^*$. For each edge $e$ connecting vertices $u$ and $v$ in $\Gamma$, there are three contributions: its weight $\wt(e)$, its Kasteleyn sign $\kappa(e)$, and a monomial $\chi(e;z,w)$ that captures how $e$ crosses the fundamental domain. This monomial is defined by \[\chi(e;z,w) = z^{\alpha(e)}w^{\beta(e)},\] 
where $\alpha(e)$ is the signed intersection number of $e$ with the vertical boundary (positive if $e$ crosses from left to right, negative if from right to left, and zero otherwise), and similarly $\beta(e)$ is the signed intersection number of $e$ with the horizontal boundary (positive if $e$ crosses from bottom to top, negative if from top to bottom, and zero otherwise). Altogether, we write \[\nu(e) = \kappa(e) \wt(e) \chi(e;z,w).\]

We denote by $K(z,w)$ the \emph{modified Kasteleyn matrix}, and its entries are expressed as \[(K(z,w))_{u,v} = \sum_{e \in \mathcal{E}(u,v)}\nu(e).\]

For example, the modified Kasteleyn matrix of the Fisher graph $F_{1,1}$, shown in Figure \ref{fig:Fisher1x1}, is as follows:

\[K(z,w) =
\begin{pmatrix}
    0 & 1 & -1 & 0 & 0 & cw \\ 
    -1 & 0 & 1 & bz & 0 & 0 \\
    1 & -1 & 0 & 0 & a & 0 \\ 
    0 & -\frac{b}{z} & 0 & 0 & 1 & -1 \\ 
    0 & 0 & -a & -1 & 0 & 1 \\ 
    -\frac{c}{w} & 0 & 0 & 1 & -1 & 0
\end{pmatrix}.\]

\medskip

Notice that we chose some edge weights to be equal to one. 
This can be done without loss of generality by applying suitable gauge transformations. More generally, whenever we have a graph that has vertex disjoint copies of triangles, we can assume that their edge weights are one.

We are particularly interested in studying the \emph{characteristic polynomial}, which is obtained as the determinant of the modified Kasteleyn matrix described above,

\[P(z,w) = \det\big(K(z,w)\big).\]

The \emph{spectral curve} is the set of all $(z,w)\in(\mathbb C^*)^2$ such that $P(z,w) = 0$.

We can compute the characteristic polynomial of the Fisher graph $F_{1,1}$,

\begin{equation}\label{char for F_11}
\begin{aligned}
P(z,w) &= a^2 + b^2 + c^2 + a^2 b^2 c^2 + a b (1 - c^2) \left(z + \frac{1}{z}\right) \\
       &+ a c (1 - b^2) \left(w + \frac{1}{w} \right)
       + b c (1 - a^2) \left(\frac{w}{z} + \frac{z}{w} \right).
\end{aligned}
\end{equation}

In the following lemma, we state some fundamental properties of the characteristic polynomials that we rely on throughout the paper. It consists of the results appearing in the proof of Lemma 13 of~\cite{BoutillierdeTiliere2010}, which describes the collections of loops contributing to the characteristic polynomial. 

We denote by $\mathcal{L}(\Gamma)$ all the collections of disjoint simple oriented loops and double edges visiting all the vertices of $\Gamma$. Note that an isolated vertex is not counted as an oriented loop. Let $\gamma \in \mathcal{L}(\Gamma)$, then $\ell(\gamma)$ is the number of loops in $\gamma$, and $\gamma(v)$ is the successor of $v \in V(\Gamma)$ in the oriented loop of $\gamma$ that contains $v$.

\begin{lem} \label{char_poly_properties}~\cite{BoutillierdeTiliere2010}
    Let $\Gamma$ be a torus graph that contains a perfect matching, and denote by $P(z, w)$ its characteristic polynomial. The following properties hold:
    \begin{enumerate}[(a)]
        \item  We can write
        \[P(z,w) = \sum_{\gamma \in \mathcal{L}(\Gamma)} (-1)^{|V(\Gamma)|-\ell(\gamma)} \prod_{e \in E(\gamma)} \nu(e).\]
        
        \item Elements $\gamma\in\mathcal L(\Gamma)$ that contain a topologically trivial loop of odd length have a zero net contribution to $P(z,w)$. 
        
        \item The non-trivial loops in $\gamma$ have parallel homology classes (either the same or opposite to each other). 
        
        \item For every $(z, w) \in (\mathbb{C^*})^2$, we have $P(z,w) = P(z^{-1}, w^{-1})$.
    \end{enumerate}
\end{lem}

A collection of disjoint oriented loops $\gamma$ can be seen as a permutation without fixed points, whereby a vertex is mapped to its successor in the loops. With this interpretation, the sign $(-1)^{|V(\Gamma)|-\ell(\gamma)}$ corresponds to the signature of the permutation.

\subsubsection*{Bipartite Kasteleyn matrix}

Suppose now that a planar graph $G = (V,E)$ is bipartite. Then, we can color its vertices black and white such that each edge connects one black vertex to one white vertex. We can denote its vertex set by $V = \{b_1, \dots, b_n, w_1, \dots, w_n\}$. Note that there should be an equal number of black and white vertices, otherwise, $G$ contains no perfect matchings. 

In this setting, if one orders the vertices as $b_1, \dots, b_n, w_1, \dots, w_n$, then the $2n \times 2n$ Kasteleyn matrix $K$ defined above is skew-symmetric. The bipartite Kasteleyn matrix $K^{\mathrm{bip}}$ is then defined as the $n \times n$ top-right block of $K$. If the graph $G$ is oriented such that it respects the clockwise odd rule, then it follows that:
\[Z(G) = |\det(K^{\mathrm{bip}})|.\]

The definition of the modified Kasteleyn matrix can also be adapted in a similar way to a bipartite torus graph $\Gamma$, and we denote it by $K^{\mathrm{bip}}(z,w)$. If an edge $b_iw_j \in E(\Gamma)$ crosses the vertical boundary of the fundamental domain such that the black vertex
is on its left (resp. right) side, then the weight $\wt(b_iw_j)$ is multiplied by $z$ (resp. $z^{-1}$). Similarly, we multiply the weight by $w$ (resp. $w^{-1}$) if the black vertex is below or above the horizontal boundary. The determinant of this modified bipartite Kasteleyn matrix is called the \emph{bipartite characteristic polynomial}:

\[P^{\mathrm{bip}}(z,w) = \det(K^{\mathrm{bip}}(z,w)).\]

This bipartite characteristic polynomial carries different combinatorial information than the characteristic polynomial $P(z,w)$. In fact, every coefficient of the polynomial $P^{\mathrm{bip}}(z,w)$ describes a dimer partition function with a prescribed homology class when it is superimposed on a fixed dimer cover. As we can see in the example of the non-bipartite Fisher graph $F_{1,1}$ (Equation \ref{char for F_11}), the terms contributing to a given monomial $z^iw^j$ do not all have the same sign. This differs from the bipartite case, where the terms of a monomial of a bipartite characteristic polynomial always have the same sign, making the bipartite case nicer to study.

\subsection{Newton polygon and zig-zag paths}

We recall here the notions of Newton polygons and zig-zag paths for bipartite graphs and then describe their natural extension to the non-bipartite setting.

\subsubsection*{Newton polygons for bipartite graphs}

A \emph{bipartite zig-zag path} on the bipartite graph $\Gamma$ is an oriented path that turns maximally right at a black vertex and maximally left at a white vertex. A bipartite torus graph $\Gamma$ is \emph{minimal} if, in its lift $\tilde\Gamma$ to the universal cover, none of its bipartite zig-zag paths is self-intersecting, and no two zig-zag paths intersect twice with the same orientation. These graphs were introduced in~\cite{GoncharovKenyon}, where their Newton polygons were studied.  

We say that a bipartite torus graph $\Gamma$ is \emph{isoradial} if, in its lift $\tilde\Gamma$ to the universal cover, none of its bipartite zig-zag paths is self-intersecting, and no two bipartite zig-zag paths intersect more than once. These graphs were studied by Kenyon and Schlenker~\cite{KenyonSchlenker}. See also the related notion of a critical map for the Ising model due to Mercat~\cite{Mercat01}. Clearly, every isoradial bipartite graph is minimal. A minimal bipartite graph is \emph{associated with} a convex integral polygon $N$ if the homology classes of its bipartite zig-zag paths in $\Gamma$ are obtained from the primitive edge vectors of $N$.

\begin{thm} [\cite{GoncharovKenyon}, Theorem 2.5] \label{polygons_graphs GK}
    For any convex integral polygon $N$, there exists a minimal bipartite graph $\Gamma$ associated with $N$. Moreover, any two minimal bipartite graphs associated with $N$ are related by a sequence of spider moves and contracting or expanding of $2$-valent vertices.
\end{thm}

In Figure \ref{fig:local_bip_moves}, we can see two types of local moves that preserve the dimer partition function. They played an important role in introducing an integrable system on minimal bipartite graphs by Goncharov and Kenyon~\cite{GoncharovKenyon}. In the proof of Theorem \ref{polygons_graphs GK}, they rely on the connection between bipartite graphs and triple crossing diagrams~\cite{DThurston}, or in different terms, plabic graphs~\cite{Pos06}.

\begin{figure}[h!]
    \centering
    \includegraphics[width=0.6\textwidth]{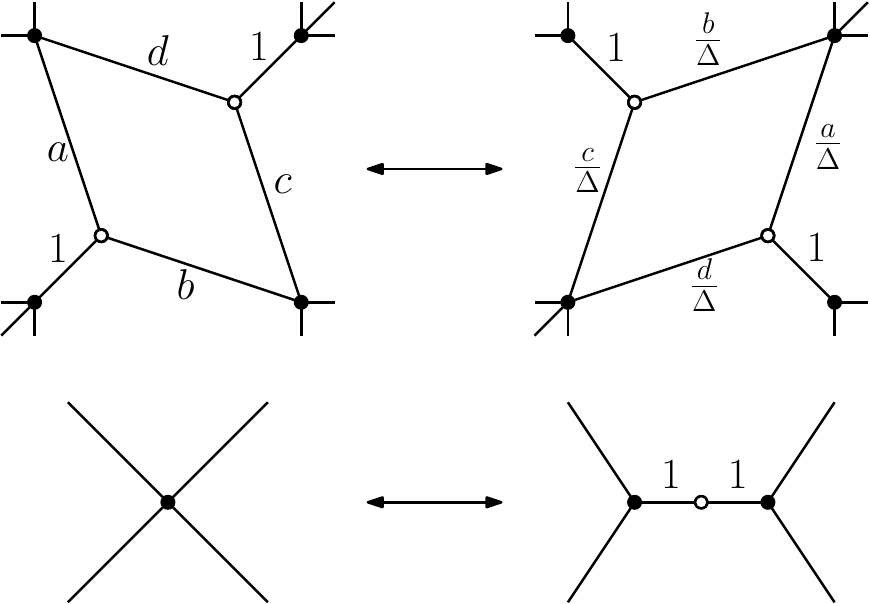}
    \caption{\label{fig:local_bip_moves} A spider move and an expansion/contraction of a degree $2$ vertex, $\Delta := ac + bd$.}
\end{figure}

\begin{rem}
    In the proof of Theorem \ref{polygons_graphs GK}, a construction on the torus is used in which one starts with straight geodesics that have homology classes equal to the primitive edge vectors describing the convex integral polygon. These geodesics are then deformed using suitable isotopies to obtain a bipartite torus graph. Note that there are certain examples of minimal bipartite graphs that can be obtained directly by translating the straight geodesics without any additional isotopies. For example, this can be done in the case of the hexagonal lattice. 
    
    However, Forsg{\aa}rd~\cite{Forsgard} found an example of a pentagon to which one cannot associate a minimal bipartite torus graph by only taking the straight geodesics with the corresponding homology classes. It remains an open question to determine which convex integral polygons can be obtained from isoradial bipartite graphs. 
\end{rem}

\begin{defi}
    The \emph{Newton polygon} of the Laurent polynomial $P(z,w)$, denoted by $N(P)$, is the convex hull of integer points $(i,j) \in \mathbb{Z}^2$ such that $z^i w^j$ has a non-zero monomial in $P(z,w)$. 
\end{defi}

Goncharov and Kenyon~\cite{GoncharovKenyon} found a correspondence between the bipartite zig-zag paths on minimal bipartite torus graphs and the Newton polygons that arise from the dimer characteristic polynomial of the given graph. 

\begin{lem}[\cite{GoncharovKenyon}] \label{bip_zig_newton}
    Let $\Gamma$ be a minimal bipartite torus graph, and let $N_{\mathrm{bip}}$ be the Newton polygon of the bipartite characteristic polynomial of $\Gamma$. There is a bijection between the multiset of homology classes of the bipartite zig-zag paths on $\Gamma$ and the primitive edge vectors on $N_{\mathrm{bip}}$.
\end{lem}

\noindent See Figure \ref{fig:hex_polygon} for an example.

\begin{figure}[h!]
    \centering
    \includegraphics[width=0.7\textwidth]{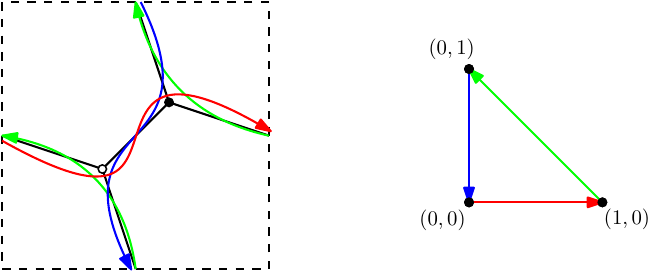}
    \caption{\label{fig:hex_polygon} Hexagonal lattice, its zig-zag paths, and the corresponding Newton polygon.}
\end{figure}

\subsubsection*{Newton polygons for non-bipartite graphs}

We want to generalize the result obtained in Lemma \ref{bip_zig_newton} to certain families of non-bipartite graphs. First, we need to adapt the definition of bipartite zig-zag paths, as it previously relied on the assumption that the graph is bipartite. 

We define a \emph{zig-zag path} as an oriented path in the graph that alternates between turning maximally left and turning maximally right. Note that for a bipartite graph, keeping or reversing the orientation of each bipartite zig-zag path yields two zig-zag paths, and every zig-zag path arises in this manner.

In the following sections, we will be particularly interested in isoradial graphs, which are defined in general in the following way.

\begin{defi} \label{gen_isoradial_graphs}
    A torus graph $\Gamma$ is \emph{isoradial} if, in its lift $\tilde\Gamma$ to the universal cover, none of its zig-zag paths is self-intersecting, and two distinct zig-zag paths that are not the reverse of each other intersect at most once.
\end{defi}

If we apply the full Kasteleyn matrix to a bipartite graph such that the first half of the vertices are white and the second half are black, then we obtain an antidiagonal block matrix. Consequently, its characteristic polynomial factors as 
\[P(z, w) = P_1(z,w)P_2(z,w),\]
where each of them is obtained as the determinant of a bipartite Kasteleyn matrix. One can observe that the Newton polygons of $P_1$ and $P_2$ are related in a way that $N_{\mathrm{bip}}(P_2) = -N_{\mathrm{bip}}(P_1)$. Thus, we obtain the Newton polygon \[N(P) = N_{\mathrm{bip}}(P_1) - N_{\mathrm{bip}}(P_1),\] where the difference is defined as the Minkowski difference of polygons. Using this correspondence, we can derive the following result directly from Lemma \ref{bip_zig_newton}.

\begin{cor} \label{NPs and homology classes}
    Let $\Gamma$ be a minimal bipartite torus graph, and let $P(z,w)$ be its characteristic polynomial obtained from the full Kasteleyn matrix. There is a bijection between the multiset of homology classes of the zig-zag paths on $\Gamma$ and the primitive edge vectors on the Newton polygon $N(P)$. 
\end{cor}

Moreover, by Lemma \ref{char_poly_properties} (d), the Newton polygon obtained from the characteristic polynomial of a graph that has a perfect matching is always centrally symmetric. This can also be seen in minimal bipartite graphs by noticing that for each zig-zag path with homology class $(a,b)$, there is a zig-zag path going through the same edges with homology $(-a,-b)$.

The zig-zag paths and the corresponding centrally symmetric Newton polygons were already studied in the context of the Ising model by George \cite{George25}.

\section{Blowing up the $3$-valent vertices} \label{trunc_graphs}

We say that a graph is \emph{$2,3$-valent} if every vertex has degree $2$ or $3$. We want to show that for $2,3$-valent torus graphs, the operation of replacing its degree $3$ vertices with triangles does not change the homology classes of its zig-zag paths. 

Before doing so, we need to set up the basis for the homology group $H_1(\mathbb{T},\mathbb{Z})$. Let $x$ and $y$ be two cycles in $\mathbb{T}$ such that their homology classes $[x]$ and $[y]$ generate $H_1(\mathbb{T},\mathbb{Z})$. From now on, using this basis, we identify the first homology group of the torus with $\mathbb{Z}^2$. Let $\mathcal{Z}(\Gamma)$ be the set of all zig-zag paths in $\Gamma$. Let $Z \in \mathcal{Z}(\Gamma)$, and denote by $[Z] \in H_1(\mathbb{T},\mathbb{Z}) \cong \mathbb{Z}^2$ the homology class of a zig-zag path $Z$. We denote by $H(\mathcal{Z}(\Gamma))$ the multiset that contains all homology classes of the zig-zag paths in $\Gamma$. 

\begin{defi} [Construction of $\Gamma_\Delta$] \label{const of gamma delta}
    Let $\Gamma$ be a $2,3$-valent graph, and let $\overrightarrow E(\Gamma)$ denote its set of oriented edges. We partition the vertices of $\Gamma$ into $V_2$ and $V_3$, respectively, the vertices of degree $2$ and $3$. Define $V'_3$ to be the collection of all pairs $(v,v')$ such that $\overrightarrow{vv'}\in \overrightarrow E(\Gamma)$ and $v\in V_3$. The vertex set of $\Gamma_\Delta$ is the disjoint union of $V_2$ and $V'_3$. 
    
    We define a projection map $\pi:V(\Gamma_\Delta) \rightarrow V(\Gamma)$, where $\pi(v)=v$ if $v\in V_2$ and $\pi(v,v')=v$ if $(v,v')\in V'_3$. For any two vertices $w,w'$ in $V(\Gamma_\Delta)$, we add an edge between $w$ and $w'$ in $\Gamma_\Delta$ if the two vertices $\pi(w)$ and $\pi(w')$ are connected by an edge in $\Gamma$. We also add an edge between any two elements of $V(\Gamma_\Delta)$ of the form $(v,v')$ and $(v,v'')$. 
\end{defi}

 For convenience, we call the edges contained in the triangles of $\Gamma_{\Delta}$ \emph{short edges}, while the remaining edges of $\Gamma_{\Delta}$ are called \emph{long edges}, which we denote by $\overrightarrow{E_{\ell}}(\Gamma_{\Delta})$. Note that there is a bijection $t: \overrightarrow{E_{\ell}}(\Gamma_{\Delta}) \rightarrow \overrightarrow{E}(\Gamma)$, between the oriented long edges of $\Gamma_{\Delta}$ and the oriented edges of $\Gamma$. We have $t(\overrightarrow{u'v'}) = \overrightarrow{uv}$, if $u' \in \Delta_{u}$ and $v' \in \Delta_{v}$, where $\Delta_{u}$ denotes the three vertices in $\Gamma_{\Delta}$ obtained from replacing $u$ with a triangle. Note that if $\deg(u) = 2$ (resp. $\deg(v) = 2$), then we take $u' = u$ (resp. $v' = v$).

Then we can define a map $\phi: \mathcal{Z}(\Gamma_{\Delta}) \rightarrow \mathcal{Z}(\Gamma)$ such that 
\[\phi(Z) = Z' = \{t(e) : e \in Z \cap \overrightarrow{E_{\ell}}(\Gamma_{\Delta})\}.\]

See Figure \ref{fig:deltazigzag} for an illustration of this map.

\begin{figure}[h!]
    \centering
    \includegraphics[width=0.7\textwidth]{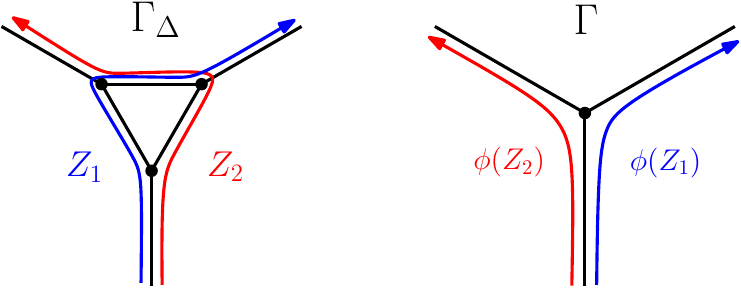}
    \caption{\label{fig:deltazigzag} The correspondence between the zig zag paths of $\Gamma_{\Delta}$ and $\Gamma$.}
\end{figure}

\begin{prop} \label{homologies_gamma_delta}
    Let $\Gamma$ be a $2, 3$-valent graph, and let $\Gamma_{\Delta}$ be the graph obtained by replacing each degree $3$ vertex with a triangle. The map $\phi$ is a bijection, and for all $Z \in \mathcal{Z}(\Gamma_{\Delta})$, we have $[Z] = [\phi(Z)]$. It follows that
    \[H(\mathcal{Z}(\Gamma)) = H(\mathcal{Z}(\Gamma_{\Delta})).\]
\end{prop}

\begin{proof}
    First of all, notice that the map $\phi$ does not change the homology of a zig-zag path $Z$. As defined, it only concatenates certain edges of $Z$. Therefore, we have $[Z] = [\phi(Z)]$, for all  $Z \in \mathcal{Z}(\Gamma_{\Delta})$. 

    We further notice that the map $\phi$ is also surjective, as any zig-zag path in $\mathcal{Z}(\Gamma)$ can be extended to a zig-zag path in $\mathcal{Z}(\Gamma_{\Delta})$ by visiting the remaining third vertex in $\Delta_u$ for every $3$-valent vertex $u$. This is consistent with the condition that zig-zag paths alternate maximally left and maximally right.
    
    It remains to show that the map is injective. Assume the contrary, there exist $Z_1$ and $Z_2$ such that $\phi(Z_1) = \phi(Z_2)$. By definition, the long edges of $Z_1$ and $Z_2$ must be the same. So, these two zig-zag paths must differ in their short edges. However, this is not possible, as by imposing which long edges are included in the zig-zag path and respecting the alternating property, for each $\Delta_u$ of the $3$-valent vertex $u$ we can choose the short edges uniquely. Thus, we proved that the map $\phi$ is bijective. As a direct consequence, it also follows that $H(\mathcal{Z}(\Gamma)) = H(\mathcal{Z}(\Gamma_{\Delta}))$.
\end{proof}

\subsection*{Newton polygons of $\Gamma$ and $\Gamma_{\Delta_S}$}

Let $S$ be a subset of $V_3$. We denote by $\Gamma_{\Delta_S}$ the torus graph obtained by replacing the subset $S$ of $3$-valent vertices of $\Gamma$ with triangles. Moreover, let $N(\Gamma)$ be the Newton polygon of the characteristic polynomial for a given torus graph $\Gamma$. We want to show by double inclusion that the Newton polygons of $N(\Gamma)$ and $N(\Gamma_{\Delta_S})$ are the same if $\Gamma$ is an isoradial bipartite graph. 

The following remark will be useful throughout the rest of the paper.

\begin{rem} [$SL_2(\mathbb{Z})$ transformation on Newton polygons] \label{sl2_NP}
    We can apply an $SL_2(\mathbb{Z})$ transformation on a Newton polygon to obtain another polygon containing a vertical side. This means that we are changing the basis for the homology group $H_1(\mathbb{T}, \mathbb{Z})$, or equivalently, choose a different fundamental domain. This is a standard procedure when working with Newton polygons associated with dimer models. See, for example, Section 3.1 of~\cite{GI24}.
\end{rem}

We now know that the fundamental domain can be chosen so that the Newton polygon contains a vertical side. We now explain how to redraw the graph within such a fundamental domain without changing its characteristic polynomial, and consequently without changing its Newton polygon.

\begin{lem} \label{vertical zig-zag path}
    Let $\Gamma$ be a torus graph that contains a zig-zag path $Z$ with a homology $(0,1)$. Then we can choose a fundamental domain, without changing the homology classes of the zig-zag paths of $\Gamma$, such that the vertical boundary of the fundamental domain crosses every edge of $Z$ once and no other edges of $\Gamma$. 
\end{lem}

\begin{proof}
    We embed $\Gamma$ in the torus without edge crossings. Then we choose a fundamental domain whose vertical boundary is given by a simple loop $\gamma_Z$ visiting the midpoints of the consecutive edges of $Z$.

    By the zig-zag property (maximal left/right turns) and planarity of the embedding, the loop can be drawn such that no edge not in $Z$ can cross it. For convenience, we deform $\Gamma$ such that its fundamental domain has a rectangular shape, that is, we deform $\gamma_Z$ into a straight curve. See Figure \ref{fig:def fund domain} for an example. This remains the same graph, and the homology classes of all its zig-zag paths are unchanged.
\end{proof}

\begin{figure}[h!]
    \centering
    \includegraphics[width=0.8\textwidth]{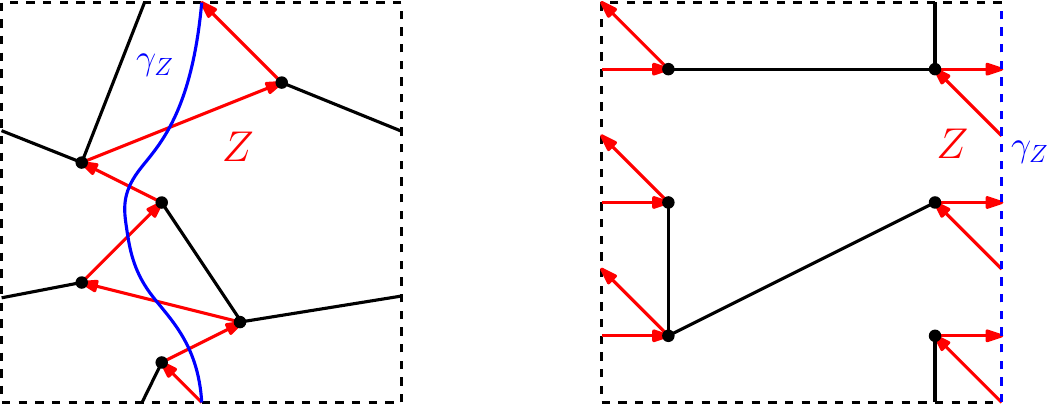}
    \caption{\label{fig:def fund domain} Left: a torus graph $\Gamma$, a zig-zag path $Z$ with homology class $(0,1)$, and a loop $\gamma_Z$. Right: the same graph $\Gamma$ obtained by deforming $\gamma_Z$ into a straight curve, and moving appropriately the vertices. This operation does not change the homology classes of any zig-zag path in $\Gamma$.}
\end{figure}

\begin{lem} \label{fundamental_domain matching}
    Let $\Gamma$ be an isoradial graph.  
    Suppose the homology classes of its zig-zag paths are given by vectors $[Z_i] = (\alpha_i, \beta_i)$ for $1 \leq i \leq n$, and assume $[Z_1] = (0,1)$. Define \[k = \sum_{i=1}^n \max(\alpha_i,0).\] 
    Then the zig-zag path $Z_1$ consists of exactly $k$ edges.
\end{lem}

\begin{proof}
    We start by applying Lemma \ref{vertical zig-zag path} and choosing the fundamental domain so that every edge of the zig-zag path $Z_1$ crosses the vertical boundary, while no other edges of $\Gamma$ do. Recall that in an isoradial graph, every edge belongs to four zig-zag paths, whose homology classes come in opposite pairs: \[(a,b), (-a,-b), (c,d), (-c,-d), \text{ with } a,b,c,d \in \mathbb{Z}.\] 

    Let $e \in E(Z_1)$. By construction, two of the zig-zag paths passing through $e$ are $Z_1$ and its opposite, with homologies $(0,1)$ and $(0,-1)$. The other two have homologies $(c,d)$ and $(-c,-d)$ for some integers $c,d$. 

    We claim that $c \neq 0$. Indeed, if $c=0$, then this zig-zag path would have homology $(0,d)$ and hence would be homologous to a vertical path. Such a path would necessarily intersect $Z_1$ more than once within the fundamental domain, which would lift to multiple intersections in the universal cover. This contradicts the isoradiality assumption, which forbids multiple intersections between zig-zag paths. Therefore, $c \neq 0$, and we may choose $(c,d)$ so that $c>0$.

    Thus, for every edge $e \in E(Z_1)$, there is a unique zig-zag path passing through $e$ whose homology has a strictly positive first coordinate. Let $\mathcal Z_+(\Gamma)$ denote the set of zig-zag paths whose homology class has a positive first coordinate. We now relate these paths to intersections with $Z_1$.

    Let $Z \in \mathcal Z_+(\Gamma)$ have homology $(\alpha,\beta)$ with $\alpha>0$. By the definition of homology, $Z$ crosses the vertical boundary exactly $\alpha$ times (counted with sign). Since the only way to cross the vertical boundary is to use an edge of $Z_1$, $Z$ uses at least $\alpha$ edges of $Z_1$. Therefore, $Z$ intersects $Z_1$ at least $\alpha$ times.

    On the other hand, by isoradiality, two zig-zag paths cannot have more than one common edge in the universal cover. This implies that each lift of $Z$ shares at most one edge with each lift of $Z_1$, and hence $Z$ shares at most $\alpha$ edges with $Z_1$ in the fundamental domain. Therefore, $Z$ shares exactly $\alpha$ edges with $Z_1$.

    Summing over all $Z \in \mathcal Z_+(\Gamma)$ with homology $(\alpha, \beta)$, we obtain that the total number of common edges between $Z_1$ and all such zig-zag paths is

    \[\sum_{Z \in \mathcal Z_+(\Gamma)} \alpha = \sum_{i=1}^n \max(\alpha_i,0) = k.\] 
\end{proof}

Lemma \ref{fundamental_domain matching} has an important role in the proof of our next result. Even though it holds for general isoradial graphs, the next result is limited to isoradial bipartite graphs, as we are also relying on Lemma \ref{bip_zig_newton}.

\begin{lem} \label{N_delta < N}
    Let $\Gamma$ be an isoradial bipartite graph and let $V_3 = \{v \in V(\Gamma): \deg(v) = 3\}$. Then
    \[N(\Gamma_{\Delta_{S}}) \subseteq N(\Gamma),\]
    where $\Gamma_{\Delta_{S}}$ denotes the graph obtained by replacing the vertices $S \subseteq V_3$ with triangles.
\end{lem}

\begin{proof}
    As discussed in Remark \ref{sl2_NP}, by applying an appropriate $SL_2(\mathbb Z)$ transformation, we can map any Newton polygon to a polygon that has at least one vertical edge vector. Using this, we can fix one side of the Newton polygon and, by an appropriate $SL_2(\mathbb Z)$ transformation, map it to a vertical side of the polygon, such that the rest of the polygon is to the left of that vertical line.

    By Lemma \ref{fundamental_domain matching}, we can embed the graph $\Gamma$ such that the number of intersections with the right boundary is equal to $k = \sum_{i=1}^n \max(\alpha_i,0)$, where $e_i = (\alpha_i, \beta_i)$ are the primitive edge vectors of $N(\Gamma)$. Consequently, the maximal matching crossing the right boundary of the fundamental domain has $k/2$ edges. We can also describe $k/2$ as the largest $z$-coordinate fitting in the Newton polygon. 

    We choose a subset $S \subseteq V_3$ and replace each vertex in it with a triangle, assuming there is no vertex on the boundary of the fundamental domain and that the triangles are small enough so that they do not touch the boundary. Let $v \in S$, and let $\Delta_v$ denote the triangle replacing $v$ in $\Gamma_{\Delta_{S}}$. If we take any collection of loops $\gamma \in \mathcal{L}(\Gamma_{\Delta_{S}})$, then it cannot contain two outgoing edges from the triangle $\Delta_v$.
    
    Assume the contrary, that $v_1, v_2 \in V(\Delta_v)$ have outgoing edges, then the third vertex $v_3$ has a double edge with an external vertex. This means that either $\overrightarrow{v_1v_2} \in \overrightarrow{E}(\gamma)$ or $\overrightarrow{v_2v_1} \in \overrightarrow{E}(\gamma)$, which is not possible. Therefore, we conclude that the maximal matching that crosses the right boundary of the fundamental domain of $\Gamma_{\Delta_S}$ remains $k/2$. The first coordinate of the homology class of $\gamma$ is at most $k/2$. So, we get $N(\Gamma_{\Delta_{S}}) \subset (-\infty, k/2] \times \mathbb{R}$.

    By performing an $SL_2(\mathbb{Z})$ transformation on each side of the Newton polygon and repeating the same argument, and then performing the inverse of each of these $SL_2(\mathbb{Z})$ transformations, we conclude that $N(\Gamma_{\Delta_{S}}) \subseteq N(\Gamma)$. 
\end{proof}

Using the same arguments as in Lemma \ref{fundamental_domain matching} and Lemma \ref{N_delta < N}, we can obtain the following result.

\begin{cor} \label{N < zigzag polygon}
    Let $\Gamma$ be a non-bipartite isoradial torus graph. Let $\mathcal{N}$ be the convex integral polygon obtained from the homology classes of the zig-zag paths of $\Gamma$. Then
    \[N(\Gamma) \subseteq \mathcal{N}.\]
\end{cor}

Recall that for bipartite graphs, the equality holds.

\begin{defi}
    If there exists $\gamma \in \mathcal{L}(\Gamma)$ with a homology class $(i,j) \in \mathbb{Z}^2$ such that every non-trivial loop in $\gamma$ has the same homology class, then $(i,j)$ is called \emph{realizable}. 
\end{defi}

The terminology of ``realizable'' is justified by the following lemma.

\begin{lem} \label{nonvanishing_coeffs}
    Let $\Gamma$ be a torus graph with complex edge weights, and let $(i,j)$ be a realizable homology class for $\Gamma$. Let $\alpha_{i,j} z^i w^j$ be a monomial in $P(z,w)$ for $(i,j) \in \mathbb{Z}^2$ , where $\alpha_{i,j}$ is a polynomial function of the edge weights. Then $\alpha_{i,j}$ is not the zero polynomial.
\end{lem}

\begin{proof}
    We take $\gamma_{i,j} \in \mathcal L(\Gamma)$ whose homology class $(i,j) \in \mathbb{Z}^2$ is realizable.
    That means $\gamma_{i,j}$ is composed of trivial loops of even length and, by Lemma \ref{char_poly_properties} (c) non-trivial loops with equal homology classes (which can be of odd length in the non-bipartite case). Since each trivial loop is of even length, we can instead look at the collection of loops $\gamma'_{i,j}$, which takes the same non-trivial loops as $\gamma_{i,j}$, but replaces each trivial loop with a union of double edges. Notice that $\gamma'_{i,j}$ has the same homology class as $\gamma_{i,j}$.
    
    We evaluate the polynomial $\alpha_{i,j}$ by assigning $\wt(e) = 0$ for all edges $e \in E(\Gamma) \setminus E(\gamma'_{i,j})$, where $E(\gamma'_{i,j})$ denotes all the edges of the collection of loops $\gamma'_{i,j}$. We consider all collections of loops with the homology class $(i, j)$ whose edge sets are subsets of the edges in $\gamma'_{i,j}$. Note that each double edge in $\gamma'_{i,j}$ must remain in the collection since all the other edges incident to its end vertices have a weight of zero. There are two ways of modifying $\gamma'_{i,j}$: one can either reverse the orientation of some non-trivial loop or transform a non-trivial loop of even length into a union of double edges. Both types of modifications change the homology class.
    
    Consequently, the only contribution to $\alpha_{i,j}$ comes from $\gamma'_{i,j}$ itself, and that contribution is non-zero by construction. If all the edge weights of $\gamma'_{i,j}$ are non-zero, the resulting product remains non-zero. This shows that the polynomial $\alpha_{i,j}$ is not identically equal to zero.
\end{proof}

As we already discussed, when a torus graph is non-bipartite, if we take one of the monomials appearing in its characteristic polynomial, the sign of the contributions is not consistent. Therefore, we do not know whether the contributions of the loops with the same homologies might cancel. In the following result, we show that they do not cancel when the monomials are taken from the extremal points of the Newton polygons.

\begin{lem} \label{ext_homology}
    Let $\Gamma$ be a torus graph. Let $(I,J) \in \mathbb{Z}^2$ be an extremal point of $N(\Gamma)$. Then $(I,J)$ is a realizable homology class.
\end{lem}

\begin{proof}
    Let $\gamma \in \mathcal{L}(\Gamma)$ with homology $(I, J)$. Then there exists $(i,j)$ such that every topologically non-trivial loop in $\gamma$ has homology either $(i,j)$ or $(-i,-j)$. Denote by $a$ the number of loops of homology $(i,j)$ and by $b$ the number of loops of homology $(-i,-j)$. Then $(I,J)=((a-b)i,(a-b)j)$. By reversing the direction of all the loops of homology $(-i,-j)$, one obtains another collection $\gamma'$, which has homology $(I',J')=((a+b)i,(a+b)j)$. Up to changing $(i,j)$ to $(-i,-j)$, we can assume that $a\geq b\geq 0$. If $b$ is non-zero, then $(I',J')$ is distinct from $(I,J)$ and a contradiction arises from the fact that $(I,J)$ can be written as the convex combination of $(0,0)$ (achievable by double edges only) and $(I',J')$.
\end{proof}

We are now ready to prove the reverse inclusion for Newton polygons. Unlike the first inclusion, this result does not require the graph to be isoradial or even bipartite.

\begin{lem} \label{N(Gamma) < N(Gamma_Delta_v)}
    Let $\Gamma$ be a torus graph containing a $3$-valent vertex $v \in V(\Gamma)$, and let $\Gamma_{\Delta_v}$ be the graph obtained by replacing the vertex $v$ with a triangle. Then
    \[N(\Gamma)\subseteq N(\Gamma_{\Delta_v}).\]
\end{lem}

\begin{proof}
    Let $\gamma \in \mathcal{L}(\Gamma)$ whose homology class corresponds to an extremal point of $N(\Gamma)$. We want to show that there exists $\gamma' \in \mathcal{L}(\Gamma_{\Delta_v})$ with the same homology. It suffices to distinguish between two cases. In the first case, the vertex $v$ is incident to a double edge in $\gamma$. Using the notation of Figure \ref{fig:gammadeltav}, this double edge is replaced in $\gamma'$ by two double edges $u_1v_1$ and $v_2v_3$. 
    
    \begin{figure}[h!]
    \centering
    \includegraphics[width=0.5\textwidth]{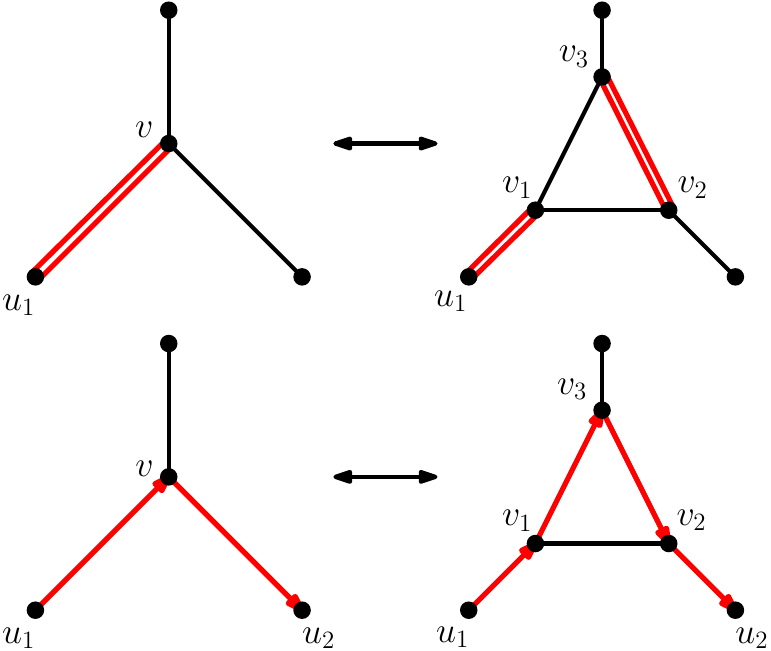}
    \caption{\label{fig:gammadeltav} The correspondence between the loops of $\Gamma$ and $\Gamma_{\Delta_v}$.}
    \end{figure}
    
    In the second case, $v$ is adjacent to two different vertices in $\gamma$, say $u_1$ and $u_2$. Then we can replace the path $u_1vu_2$ of the loop in $\gamma$ with the path $u_1v_1v_3v_2u_2$. See Figure \ref{fig:gammadeltav}. Note that $\gamma'$ indeed visits all the vertices of $\Gamma_{\Delta_v}$ and the homologies of each loop are the same between $\gamma$ and $\gamma'$.

    To conclude the proof, by Lemma \ref{ext_homology} the homology of $\gamma$ is realizable for $\Gamma$. By the construction above, the homology of $\gamma'$ remains realizable, as each loop of $\gamma$ is replaced by a loop in $\gamma'$ on $\Gamma_{\Delta_v}$ with the same homology. Therefore, by Lemma \ref{nonvanishing_coeffs}, we conclude that the coefficients corresponding to the extremal points of $N(\Gamma)$ do not vanish in the characteristic polynomial of the graph $\Gamma_{\Delta_v}$.
\end{proof}

Using Lemmas \ref{N_delta < N} and \ref{N(Gamma) < N(Gamma_Delta_v)}, we obtain the following result.

\begin{thm} \label{equal newton polygons delta}
    Let $\Gamma$ be an isoradial bipartite graph. Then
    \[N(\Gamma_{\Delta_{S}}) = N(\Gamma),\]
    where $\Gamma_{\Delta_{S}}$ denotes the graph obtained by replacing the vertices $S \subseteq V_3 = \{v \in V(\Gamma): \deg(v) = 3\}$ with triangles.
\end{thm}

\begin{rem}
    It may seem that the statement in Theorem \ref{equal newton polygons delta} holds in general. However, we can create an example of a graph, shown in Figure \ref{fig:nonexample}, where replacing a vertex by a triangle changes the Newton polygon. We can check that this graph is not isoradial. On the other hand, it may be possible to generalize the statement by replacing isoradial with minimal bipartite graphs.
\end{rem}

\begin{figure}[h!]
    \centering
    \includegraphics[width=0.8\textwidth]{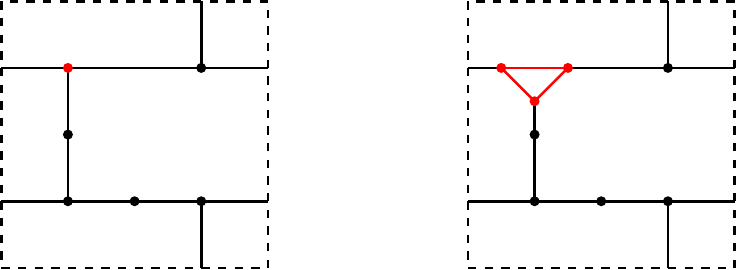}
    \caption{\label{fig:nonexample} The graph on the left contains no collection of loops with total homology class $(2,0)$, while the graph on the right does.}
\end{figure}

Recall that we write $\Gamma_{\Delta}$ if we take $S = V$ in $\Gamma_{\Delta_S}$. The following result follows from Proposition \ref{homologies_gamma_delta} and Theorem \ref{equal newton polygons delta}.

\begin{cor}
    Let $\Gamma$ be an isoradial bipartite $3$-valent graph. Then the Newton polygon $N(\Gamma_\Delta)$ is obtained from the homology classes of the zig-zag paths of $\Gamma_\Delta$.
\end{cor}

\section{Corner graphs} \label{corner_graphs}

We define a \emph{corner} as a pair of a vertex and a face that are adjacent. Each corner $(v,f)$ has four neighboring corners: the previous/next corner around $v$ and the previous/next corner around $f$. 

\begin{defi}
    For a given torus graph $\Gamma$, we obtain the corner graph $C_{\Gamma}$ by placing a vertex for each corner and adding an edge between any two neighboring corners. Note that $C_{\Gamma}$ is a $4$-valent graph. 
\end{defi}

This construction was mentioned in the context of the Ising model by Chelkak~\cite{Chelkak}. Kenyon, Sun, and Wilson~\cite{KenyonSunWilson} considered the corner graph obtained from the hexagonal lattice. They showed that for a special choice of edge weights, the characteristic polynomial of the obtained graph is reducible, and it intersects the unit torus $(\mathbb C^*)^2$ in the same way as the characteristic polynomial of the Fisher graph.

\subsection{Homology classes of zig-zag paths}

Recall that $H(\mathcal{Z}(\Gamma))$ denotes the multiset containing all the homology classes of zig-zag paths in $\Gamma$. Let $\tau^\ell, \tau^r: \overrightarrow{E}(\Gamma) \rightarrow \overrightarrow{E}(C_\Gamma)$ be the maps defined on oriented edges by \[\tau^\ell(\overrightarrow{uv}) = \overrightarrow{c'_u c'_v}, \qquad \tau^r(\overrightarrow{uv}) = \overrightarrow{c''_u c''_v},\] where $c'_u = (u, f')$, $c'_v = (v, f')$, and $c''_u = (u, f'')$, $c''_v = (v, f'')$. Here, $f'$ and $f''$ are the two faces of $\Gamma$ incident to the edge $uv$, such that $f'$ (resp. $f''$) lies to the left (resp. right) face when traversing the edge from $u$ to $v$. 

We further define maps $\phi^\ell, \phi^r: \mathcal{Z}(\Gamma) \rightarrow \mathcal{Z}(C_\Gamma)$ as follows. For a zig-zag path $Z$ in $\Gamma$, the path $\phi^\ell(Z) = Z'$ (resp. $\phi^r(Z) = Z''$) is the unique zig-zag path in $C_\Gamma$ whose oriented edges contain $\tau^\ell(e)$ (resp. $\tau^r(e)$) for every $e \in \overrightarrow{E}(Z)$. To ensure that $Z'$ and $Z''$ are well-defined zig-zag paths, it may be necessary to insert additional edges. More precisely, suppose that $Z$ makes a right turn at a vertex $v$, traversing the vertices $u \to v \to w$. Then $Z'$ is obtained by inserting the oriented edges $\overrightarrow{c'_v c''_v}$ and $\overrightarrow{c''_v c'''_v}$, where $c'''_v = (v,f''')$, and $f'''$ denotes the face distinct from $f''$ that is incident to the edge $vw$. 

On the other hand, no additional edges are required to obtain the zig-zag path $\phi^r(Z) = Z''$ when $Z$ turns right at $v$. The case where $Z$ makes a left turn at $v$ is treated analogously by symmetry. See Figure \ref{fig:rectconst} for an illustration of how $Z'$ and $Z''$ are obtained.

\begin{figure}[h!]
    \centering
    \includegraphics[width=0.8\textwidth]{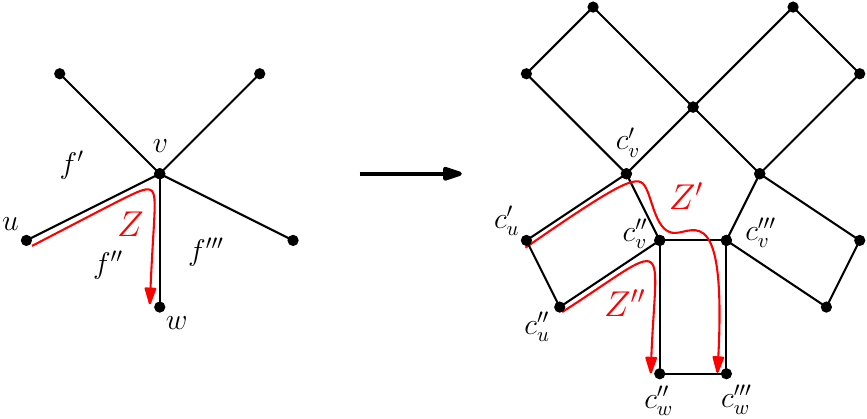}
    \caption{\label{fig:rectconst} A zig-zag path $Z$ in $\Gamma$ gives rise to two zig-zag paths $Z'$ and $Z''$ in the corner graph $C_{\Gamma}$.}
\end{figure}

\begin{prop} \label{homologies C_Gamma}
    Let $\Gamma$ be a torus graph. Then the maps $\phi^\ell$ and $\phi^r$ are injective, their images are disjoint, and $\phi^\ell(Z(\Gamma)) \cup \phi^r(Z(\Gamma)) = Z(C_\Gamma)$. Moreover, for all $Z \in Z(\Gamma)$, we have
    \[[Z] = [\phi^\ell(Z)] = [\phi^r(Z)].\]
\end{prop}

\begin{proof}
    First, we check that the images of $\phi^\ell$ and $\phi^r$ are injective and disjoint. This clearly holds, as the images of the maps $\tau^\ell$ and $\tau^r$ are also injective and disjoint. Moreover, notice that each zig-zag path in $Z(C_\Gamma)$ can be sent back to the zig-zag path in $Z(\Gamma)$ by merging all the rectangles in $C_\Gamma$ into edges of $\Gamma$. Therefore, $\phi^\ell(Z(\Gamma)) \cup \phi^r(Z(\Gamma)) = Z(C_\Gamma)$.

    Finally, we can observe that this inverse map does not change the homology class of the zig-zag path, which completes the proof.
\end{proof}

\begin{lem} \label{sp forest to span tree}
    Let $G$ be a connected graph. Any spanning forest $F$ of $G$ can be extended to a spanning tree of $G$.
\end{lem}

\begin{proof}
    Suppose $F$ consists of $k \ge 2$ disjoint trees (if $k = 1$, we are done). Because $G$ is connected, for any two distinct components of $F$, there exists a path in $G$ connecting a vertex in one component to a vertex in the other. Along this path, there must be an edge whose endpoints lie in two different components of $F$. This edge cannot belong to $F$, so it lies in the complement of $F$.

    Adding this edge to $F$ merges two components into one, producing a spanning forest with $k-1$ trees. Repeating this procedure $k-1$ times, we obtain a spanning tree of $G$.
\end{proof}

\begin{lem} \label{loops without double edges on corner graphs}
    Let $\Gamma$ be a torus graph. Let $\gamma \in \mathcal{L}(\Gamma)$ such that $[\gamma] = (a,b) \neq (0,0)$ is realizable. Then there exists $\gamma_c \in \mathcal{L}(C_\Gamma)$ such that $[\gamma_c] = (2a, 2b)$ is realizable, and $\gamma_c$ contains no double edges.
\end{lem}

\begin{proof}
    Without loss of generality, suppose that the only trivial loops of $\gamma$ are double edges. Suppose that $\gamma$ contains $k \geq 1$ non-trivial loops in $\gamma$, denoted by $\{\gamma_1, \dots, \gamma_k\}$. Removing all the vertices and edges of all the non-trivial loops from $\Gamma$ decomposes the fundamental domain into $k$ cylinder graphs, which we denote by $\Gamma_i$ for $1 \leq i \leq k$, where $\Gamma_i$ lies between $\gamma_i$ and $\gamma_{i+1}$. The double edges in $\gamma$ induce a perfect matching $M_i$ for each $\Gamma_i$. By Lemma \ref{sp forest to span tree}, we can obtain a spanning tree from $M_i$ for each connected component of $\Gamma_i$. Consequently, this yields a spanning forest $F_i$ of $\Gamma_i$.

    From the spanning forests $F_i$ and non-trivial loops $\gamma_i$, we want to construct a cycle-rooted spanning forest, whose cycles are precisely the non-trivial loops $\gamma_i$.
    Since $\Gamma$ is a torus graph, it is in particular connected. Similarly to Lemma \ref{sp forest to span tree}, for each connected component of $F_i$, there exists an edge connecting it to a vertex on either $\gamma_i$ or $\gamma_{i+1}$ (if both are possible, we choose only one arbitrarily). Since $\Gamma$ is connected, we can repeat this until we obtain two cycle-rooted trees that visit all the vertices of $\Gamma_i$. Repeating this construction for every $i \in \{1, \dots, k\}$ yields the desired cycle-rooted spanning forest. See the upper part of Figure \ref{fig:spanning_forest_double_loops} for an example of how a cycle-rooted spanning forest can be obtained from $\gamma$.

    Finally, to each cycle-rooted tree containing $\gamma_i$ in our construction, we can associate two non-trivial loops in $C_{\Gamma}$ with the same homology as the non-trivial loop $\gamma_i$. These two loops follow the contour around the cycle-rooted tree. See Figure \ref{fig:spanning_forest_double_loops} for an illustration.
    
    More precisely, let $\mathcal{T}_i$ be the cycle-rooted tree containing $\gamma_i$. We partition its edges into $E^l_i$ and $E^r_i$, where $E^l_i$ (resp. $E^r_i$) are the edges appearing between the loops $\gamma_{i-1}$ and $\gamma_i$ (resp. $\gamma_i$ and $\gamma_{i+1}$). We handle the edge in $E^l_i$ first.
    For all $uv \in E^\ell_i$, we choose the edges $\tau^\ell(\overrightarrow{uv}) = \overrightarrow{c'_u c'_v}$ and $\tau^\ell(\overrightarrow{vu}) = \overrightarrow{c''_v c''_u}$ to be contained in the loop on $C_\Gamma$. \footnote{This is the same map as defined earlier in the section.} In order to extend these edges into a loop, we add the necessary edges that arise from the corners associated with the same vertex. We start with a corner that contains an ingoing edge and connect it with the first corner appearing in the clockwise direction around a polygon that replaces the given vertex.
\end{proof}

\begin{figure}[h!]
    \centering
    \includegraphics[width=0.9\textwidth]{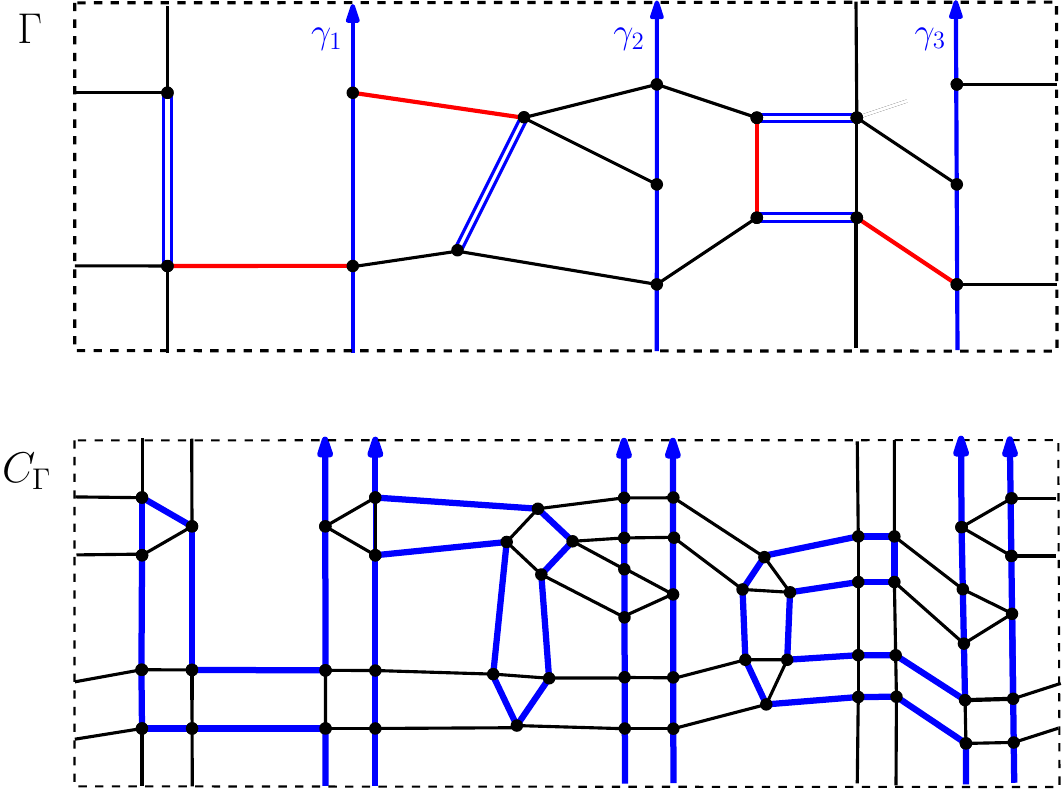}
    \caption{\label{fig:spanning_forest_double_loops} Up: a collection of loops $\gamma$ in blue with homology $(a,b)$, extended with red edges into a cycle-rooted spanning forest. Down: the corresponding collection of loops $\gamma_c$ without double edges and homology $(2a,2b)$.}
\end{figure}

\subsection{Newton polygons of corner graphs}

We now describe the relationship between the Newton polygons of a torus graph $\Gamma$ and its corner graph $C_\Gamma$.

\begin{lem} \label{2N Gamma < N C_Gamma}
    Let $\Gamma$ be a torus graph. Then
    \[2N(\Gamma) \subseteq N(C_{\Gamma}).\] 
\end{lem}

\begin{proof}
    We take $(\alpha, \beta) \in \mathbb{Z}^2$ to be an extremal point of $N(\Gamma)$. Let $\gamma \in \mathcal{L}(\Gamma)$, such that $(\alpha, \beta)$ is its homology class, which by Lemma \ref{ext_homology} is realizable.
    
    Applying further Lemma \ref{loops without double edges on corner graphs}, we can obtain a loop $\gamma_c \in \mathcal{L}(C_\Gamma)$, with a realizable homology class $(2 \alpha, 2 \beta)$. By Lemma \ref{nonvanishing_coeffs}, we get that the coefficient of the monomial $z^{2\alpha}w^{2\beta}$ does not vanish in the characteristic polynomial of $C_\Gamma$. 
    Therefore, $(2\alpha, 2\beta) \in N(C_\Gamma)$. By repeating the same argument at all the extremal points, the conclusion follows.
\end{proof}

We proceed to prove the opposite direction for a more restricted family of torus graphs. We start with an elementary topological observation.

\begin{obs} \label{at most 2 loops}
    Let $\gamma \in \mathcal{L}(C_\Gamma)$. Let $v \in V(\Gamma)$ be a vertex of degree $d$ and $\{c_1, \dots, c_d \} \subseteq V(C_\Gamma)$ be the $4$-valent vertices that replace the vertex $v$ in the construction of the corner graph. Then in $\gamma$, there can be at most two different non-trivial loops visiting the corners $\{c_1, \dots, c_d \}$, otherwise, they must intersect each other, which is not allowed for the collections of loops in $\mathcal{L}(C_\Gamma)$.
\end{obs}

\begin{lem} \label{N(C_Gamma) < 2 N(Gamma)}
    Let $\Gamma$ be an isoradial bipartite graph. Then
    \[N(C_{\Gamma}) \subseteq 2 N(\Gamma).\] 
\end{lem}

\begin{proof}
    We proceed similarly to the proof of Lemma \ref{N_delta < N}. Up to an $SL_2(\mathbb{Z})$ transformation, we may assume that $N(\Gamma)$ has a vertical side. By Lemma \ref{fundamental_domain matching}, the graph $\Gamma$ admits an embedding for which the number of intersections with the right boundary is $k = \sum_{i=1}^n \max(\alpha_i,0)$, where $e_i = (\alpha_i, \beta_i)$ are the primitive edge vectors of $N(\Gamma)$. Consequently, the maximal matching crossing the right boundary of the fundamental domain consists of $k/2$ edges.
    
    After replacing $\Gamma$ by its corner graph $C_{\Gamma}$, Observation \ref{at most 2 loops} implies that the maximal matching crossing the vertical boundary has $k$ edges. It follows that \[N(C_{\Gamma}) \subseteq [-k, k] \times \mathbb{R}.\] Applying the same argument to each side of the Newton polygon, we conclude that $N(C_{\Gamma}) \subseteq 2 N(\Gamma)$.
\end{proof}

Using Proposition \ref{homologies C_Gamma} together with Lemmas \ref{2N Gamma < N C_Gamma} and \ref{N(C_Gamma) < 2 N(Gamma)}, we obtain the following.

\begin{cor}
    Let $\Gamma$ be an isoradial bipartite graph, and let $C_\Gamma$ be its corner graph. Then the primitive edge vectors of the Newton polygon $N(C_\Gamma)$ coincide with the homology classes of the zig-zag paths in $C_\Gamma$.
\end{cor}

\subsection{Other non-bipartite graphs}

Here, we would like to briefly discuss yet another family of non-bipartite graphs. These graphs are known as Fisher graphs, and they are obtained from an initial graph $\Gamma$ by replacing each vertex of degree $d \geq 3$ with a decoration of $2d$ vertices, which contains a chain of $d$ triangles, as shown in Figure \ref{fig:zigzagsgenfisher}. This is one of the possible ways to relate an Ising model of the initial graph with a dimer model on the decorated graph, introduced by Fisher \cite{Fisher66}. 

\begin{figure}[h!]
    \centering
    \includegraphics[width=0.8\textwidth]{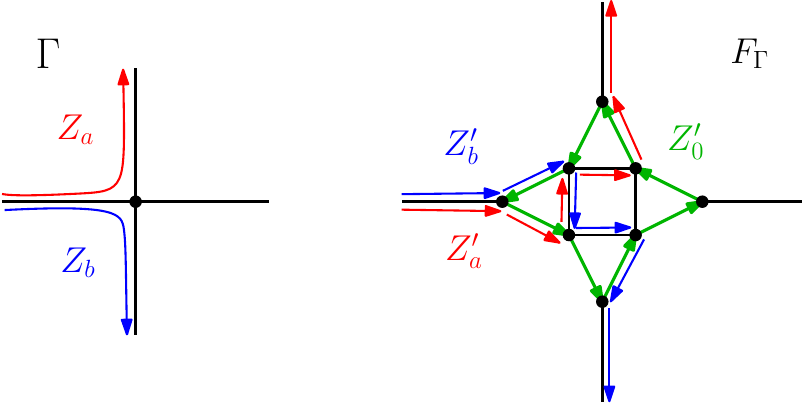}
    \caption{\label{fig:zigzagsgenfisher} The zig-zag paths on the graph $\Gamma$, and the corresponding zig-zag paths in $F_\Gamma$, containing an extra zig-zag path of $z'_0$ of homology zero.}
\end{figure}

For these graphs, we also discuss their Newton polygons and zig-zag paths. We observe an intermediate phenomenon where the Newton polygon of the Fisher graph, denoted by $F_\Gamma$, contains $N(\Gamma)$ and is contained in $2N(\Gamma)$. However, as we are mainly relying on the techniques used in previous sections, we omit the proofs of these results.

Similarly to what we proved in Proposition \ref{homologies_gamma_delta}, we can observe in the case of Fisher graphs that there is a map between the zig-zag paths of $\Gamma$ and the zig-zag paths of $F_\Gamma$. However, we notice something slightly different happening for this family of graphs. In fact, this map is no longer bijective. As noted in Figure \ref{fig:zigzagsgenfisher}, the zig-zag paths $Z_a$ and $Z'_a$ (resp. $Z_b$ and $Z'_b$) have the same homologies. On the other hand, for each vertex in $\Gamma$, there are two new zig-zag paths of homology zero appearing in $F_\Gamma$. We denote one such zig-zag path in Figure \ref{fig:zigzagsgenfisher} by $Z'_0$. The other is obtained by reversing the orientation of $Z'_0$. Therefore, we can state the following analogue of Proposition \ref{homologies_gamma_delta}.

\begin{prop}
    Let $\Gamma$ be a torus graph of minimum degree $3$, and let $F_\Gamma$ be its Fisher graph. Then
    \[H(\mathcal{Z}(F_\Gamma)) = H(\mathcal{Z}(\Gamma)) \cup \{(0,0)^{2n}\},\]
    where $n = |V(\Gamma)|.$
\end{prop}

Similarly to what we discussed in Observation \ref{at most 2 loops} and Lemma \ref{N(C_Gamma) < 2 N(Gamma)}, we can show that $N(F_\Gamma) \subseteq 2N(\Gamma)$. Using the ideas from the proof of Lemma \ref{N(Gamma) < N(Gamma_Delta_v)}, we can show that $N(\Gamma) \subseteq N(F_\Gamma)$. In this case, as we are allowing arbitrary degrees $d \geq 3$, it is slightly more delicate to associate a collection of loops in $F_\Gamma$ that has the same homology as an extremal collection of loops of $\Gamma$. However, by carefully choosing the double edges, it is possible to obtain such a collection of loops. 

Note that the following result assumes that the edge weights of $F_{\Gamma}$ are generic. This is not the case for the usual Fisher correspondence, where all edges of the decoration are assigned weight one.

\begin{prop}
    Let $\Gamma$ be an isoradial bipartite graph. Then
    \[N(\Gamma) \subseteq N(F_\Gamma) \subseteq 2N(\Gamma).\]
\end{prop}

In Figure \ref{fig:sqfisherlattice}, we show an example of a graph for which the Newton polygon is strictly between the two bounds.

\begin{figure}[h!]
    \centering
    \includegraphics[width=0.9\textwidth]{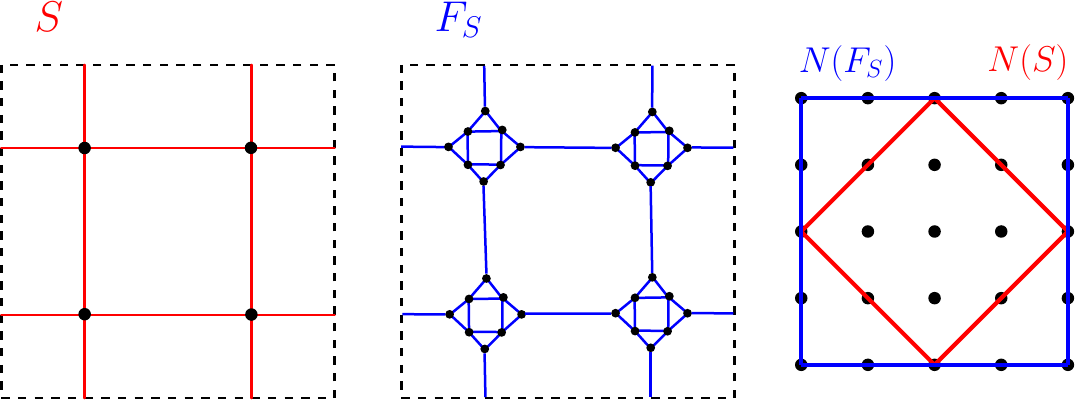}
    \caption{\label{fig:sqfisherlattice} The square lattice $S$, the graph $F_{S}$ and the comparison of their Newton polygons.}
\end{figure}

\section{Marginal polynomials} \label{marginal_polys}

In this section, we study the roots of the polynomials that correspond to the boundary sides of the Newton polygons, which we call \emph{marginal polynomials}. For minimal bipartite graphs, George, Goncharov, and Kenyon~\cite{GGK23} showed that these polynomials have a universal property. In fact, they are always real-rooted, and up to a sign, their roots can be expressed as the alternating product of the edge weights of the zig-zag paths with a corresponding homology class. Here, we study several examples of non-bipartite graphs and show the real-rootedness of their marginal polynomials.

We first focus on specific examples of the triangular lattice and the Fisher graph of the hexagonal lattice for which we can compute precisely the roots of their marginal polynomials. Then, we prove a general result for a large family of non-bipartite graphs whose Newton polygons we discussed in Section \ref{trunc_graphs}. 

\subsection{Triangular lattice}

We first consider the triangular lattice, whose fundamental domain has size $m \times n$ and is denoted by $T_{m,n}$. This graph has a perfect matching if at least one of $m$ and $n$ is even. In Figure \ref{fig:T2x2andNP2x2}, there is a fundamental domain of size $2 \times 2$ together with its edge weights and Kasteleyn orientation. 

The Newton polygon for the triangular lattice $T_{m,n}$ has a hexagonal shape, with vertices $(m, 0), (0, n), (m, n)$ and their opposites $(-m, 0), (0, -n), (-m, -n)$. The Newton polygon of the $2 \times 2$ triangular lattice is shown in Figure \ref{fig:T2x2andNP2x2}. 

\begin{figure}[h!]
    \centering
    \includegraphics[width=0.75\textwidth]{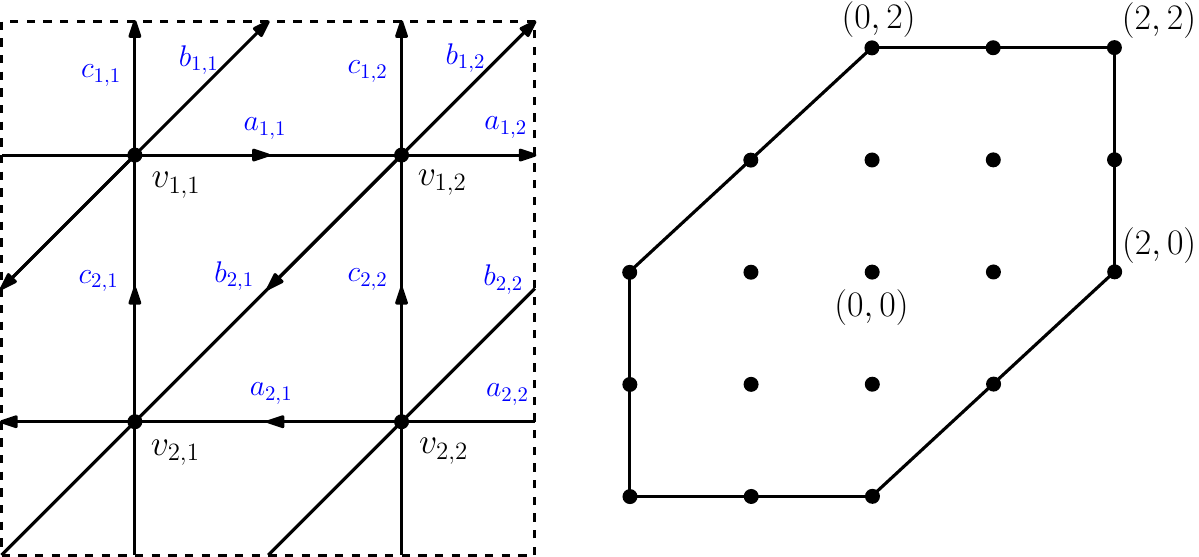}
    \caption{\label{fig:T2x2andNP2x2} Triangular lattice $T_{2,2}$ and its Newton polygon.}
\end{figure}

While all the $c$-edges carry the same Kasteleyn orientation, the orientation of the $a$-edges and of the $b$-edges depends on the parity of their first coordinate: to the right if it is odd, and to the left if it is even. We are interested in marginal polynomials that correspond to the boundaries of the Newton polygon. 

First, we look at the vertical side of the Newton polygon that connects the points $(m,0)$ and $(m,n)$. The polynomial obtained by summing the monomials that correspond to this side of the Newton polygon can be written as $z^m P_R(w)$. Notice that the polynomial $P_R$ has degree $n$. Similarly, we can examine the polynomial that corresponds to the top horizontal side of the Newton polygon. We denote this polynomial by $P_U(z)$, and it has degree $m$.

\begin{prop} \label{rootsT}
    Let $m,n \in \mathbb{N}$ and assume that $m$ is even. The polynomials $P_R(w)$ and $P_U(z)$ have all real roots, and each root can be written as 
    \[r_i = \prod_{1 \leq j \leq m} \frac{a_{j,i}}{b_{j,i}}, \text{ for } 1 \leq i \leq n,\]
    \[u_j = (-1)^{(j-1)n} \prod_{1 \leq i \leq n} \frac{c_{j,i}}{b_{j,i}} ,\text{ for } 1 \leq j \leq m.\]
\end{prop}

\begin{proof}
    Recall from Lemma \ref{char_poly_properties} (a) that each term in $P(z,w)$ represents a collection of oriented loops on the graph $T_{m,n}$. To obtain a term in which the $z$-variable has degree $m$, every vertex in $\{v_{1,n}, \dots, v_{m,n}\}$ must have an outgoing edge crossing the right boundary of the fundamental domain. In particular, the vertex $v_{1,n}$ must be connected either to $v_{1,1}$ or to $v_{m,1}$.
    
    If $v_{1,n}$ is connected to $v_{1,1}$, then $v_{2,n}$ is connected to $v_{2,1}$ because otherwise $v_{1,1}$ would have two ingoing edges. Continuing the argument inductively, all the edges crossing the right boundary must be parallel. Analogously, for each column, we can independently choose either all $a$-edges or all $b$-edges. This allows us to express the polynomial $P_R(w)$ as a product of linear factors, from which we can extract its roots.
    
    There are two contributions that affect the signs of these factors. The first comes from the Kasteleyn orientation. In our case, selecting all $a$-edges in the first column yields the product $(-1)^{m/2}a_{1,1}a_{2,1}...a_{m,1}$, while selecting all $b$-edges yields $(-1)^{m/2}b_{1,1}b_{2,1}...b_{m,1}w$. The second contribution comes from the number of loops in a given collection, by Lemma \ref{char_poly_properties} (a). Let $k \in \{0, 1, \dots, n\}$, and consider the monomial $z^m w^k$ lying on the vertical side of the Newton polygon that we are considering. Since the collections of loops are always assumed to be non-intersecting, Lemma \ref{char_poly_properties} (c) implies that all the non-trivial loops must be parallel. Moreover, the homology class of each loop is $(m/d, k/d)$, where $d = \gcd(m,k)$, and therefore there are exactly $d$ such loops, and by Lemma \ref{char_poly_properties} (a), the total sign of this collection of loops is $(-1)^{mn-d}$.
    
    Since $m$ is even, the integers $d$ and $k$ have the same parity, and thus $(-1)^{mn-d} = (-1)^k$. Since $k$ is the number of columns of $b$-edges appearing in that collection of loops, we conclude that whenever the $b$-edges from a given column are included, the global sign is multiplied by $-1$. It follows that the polynomial can be written as the product

    \begin{align*}
        P_R(w) &= \prod_{j=1}^n((-1)^{m/2}a_{1,j}a_{2,j}...a_{m,j} - (-1)^{m/2}b_{1,j}b_{2,j}...b_{m,j}w) \\
    &= (-1)^{mn/2} \prod_{j=1}^n(a_{1,j}a_{2,j}...a_{m,j} -b_{1,j}b_{2,j}...b_{m,j}w).
    \end{align*}
    
    The roots of $P_R(w)$ are therefore of the desired form.

    To determine the roots of the polynomial $P_U(z)$, we proceed similarly. Since the Kasteleyn orientation is not symmetric in the $z$-direction and $w$-direction, we need to take care of the signs of the roots. For each row, we have a choice either to take all $b$-edges or all $c$-edges. The product of $b$-edges in the row $j$ is equal to $(-1)^{(j-1)n}b_{j,1}b_{j,2} \dots b_{j,n}z$, because the orientation of $b$-edges alternates in each row. The product of $c$-edges is equal to $c_{j,1}c_{j,2} \dots c_{j,n}$, because the $c$-edges are always oriented upwards. By an analogous argument as for $P_R(w)$, we notice that the global sign of the permutation is multiplied by $-1$ whenever we take all $b$-edges. So, the polynomial $P_U$ can be written as the product

    \[P_U(z) = (-1)^{mn/2} \prod_{j = 1}^m(c_{j,1}c_{j,2}...c_{j,n} - (-1)^{(j-1)n}b_{j,1}b_{j,2}...b_{j,n}z).\]
\end{proof}

A combinatorial way to see the roots of $P_R$ is to notice that each root is actually the alternating product of the edge weights that correspond to the zig-zag path with homology class $(0,1)$, which is the direction of the vector corresponding to the vertical side of the Newton polygon. The triangular lattice is symmetric in terms of $a, b, c$-edges. So, we can repeat the same argument for the remaining sides of the Newton polygon. 

We can prove an analogous result for a larger family of non-bipartite graphs.

\begin{thm}
    Let $\Gamma$ be an isoradial torus graph such that its Newton polygon coincides with the polygon described by the homology classes of the zig-zag paths of $\Gamma$. Then the roots of the marginal polynomials associated with $N(\Gamma)$, up to a sign, can be described as the alternating product of the edge weights of a zig-zag path with the prescribed homology class.
\end{thm}

From Corollary \ref{N < zigzag polygon}, we know that the Newton polygon is contained in the polygon obtained from the homologies of zig-zag paths. However, we do not know whether the opposite direction holds.

\subsection{Hexagonal Fisher graphs}

We now consider Fisher graphs that are obtained from the hexagonal lattice by replacing each vertex with a triangle. Again, we are interested in general $m \times n$ Fisher graphs, which we denote by $F_{m, n}$. The smallest fundamental domain is shown in Figure \ref{fig:Fisher1x1} and has six vertices. Therefore, $F_{m,n}$ has $6mn$ vertices in total. Spectral curves of Fisher graphs were previously studied by Li~\cite{ZLi}, Cimasoni, and Duminil-Copin~\cite{CimasoniDuminilCopin}.

The Newton polygon of $F_{m,n}$ is the same as for the triangular lattice $T_{m,n}$, but reflected over the $w$-coordinate. See Figure \ref{fig:Fisher2x3NP}. Again, our goal is to determine the roots of the polynomials that correspond to each side of the Newton polygon. We start with the vertical side of the Newton polygon, and we denote by $P_R(w)$ the polynomial that collects all the monomials of the form $z^m w^j$ where $-n \leq j \leq 0$. Before we proceed to study the marginal polynomials, we need an additional statement about the loops that correspond to the extremal points of the Newton polygon.

\begin{figure}[h!]
    \centering
    \includegraphics[width=0.85\textwidth]{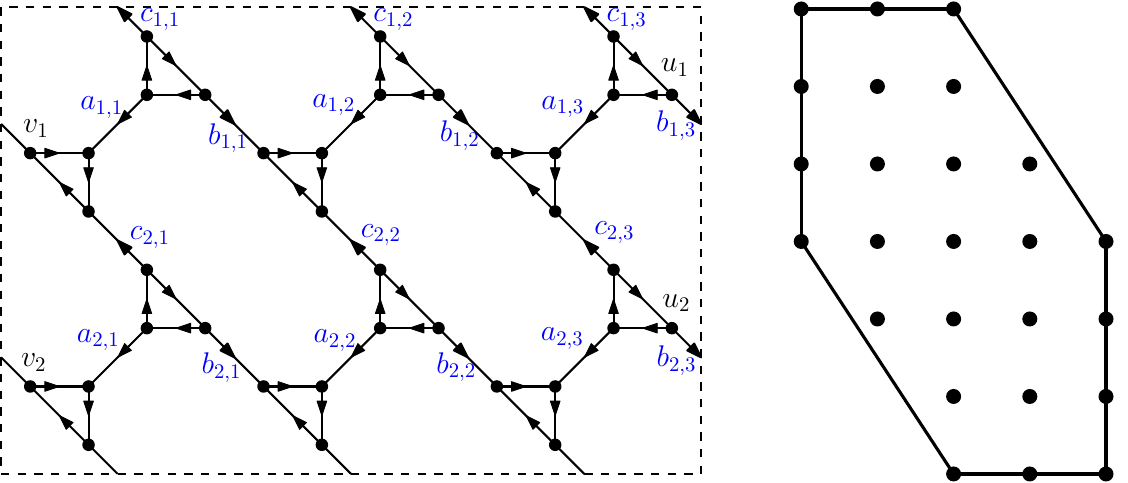}
    \caption{\label{fig:Fisher2x3NP} The Fisher graph $F_{2,3}$ and its Newton polygon.}
\end{figure}

\begin{lem} \label{no double edges}
    Let $\gamma \in \mathcal{L}(F_{m,n})$ be a collection of loops with a realizable homology $(m,i)$ for any $-n \leq i \leq 0$. Then $\gamma$ satisfies the following properties:
    \begin{enumerate} [(a)]
        \item The non-trivial loops in $\gamma$ traverse all $b$-edges of $F_{m,n}$;
        \item The only trivial loops in $\gamma$ are double edges traversing an $a$- or $c$-edge;
        \item For every column of $F_{m,n}$, the non-trivial loops in $\gamma$ traverse either all $a$-edges or all $c$-edges.
    \end{enumerate}
\end{lem}

\begin{proof}
    \begin{enumerate}[(a)]
        \item We start by partitioning the set of $b$-edges into $\{B_1, \dots, B_n\}$, where $B_k = \{b_{1,k}, \dots, b_{m,k}\}$ for all $1 \leq k \leq n$.  In order to obtain a monomial that contains $z^m$, we need all the edges of $B_n$ in the collection of oriented loops. We denote by $u_iv_i$ the edge whose weight is equal to $b_{i,n}$, such that $\gamma$ traverses all the edges $u_i \to v_i$. See Figure \ref{fig:Fisher2x3NP}. Then, we want to show that consequently, all the edges of $B_k$ are contained in the loops of $\gamma$ for all $1 \leq k \leq n-1$.

        Notice that by deleting the edge sets $B_k$ and $B_n$, the graph becomes disconnected. Moreover, notice that in $\gamma$ from each $v_i$ for $1 \leq i \leq m$, there is a path that goes to some $u_j$, for $1 \leq j \leq m$. Every such path has to go through at least one of the edges in $B_k$. Since the loops are disjoint and there are exactly $m$ paths joining $(v_i)_{1 \leq i \leq m}$ and $(u)_{1 \leq j \leq m}$, it follows that all the edges from $B_k$ appear in $\gamma$. So, we can conclude that all $b$-edges are visited by $\gamma$. 
        \item Suppose there is a trivial loop of length greater than $2$ in $\gamma$. From part (a), we know that such a loop does not visit any of the vertices covered by $b$-edges. If we delete all the vertices incident to $b$-edges, we obtain a graph that consists of $n$ vertex-disjoint cycles. The only loops contained in this graph are the non-trivial loops.
        
        Suppose that $\gamma$ contains a double edge traversing an edge of some triangle. Observe that every triangle contains a vertex incident to a $b$-edge. Hence, the double edge must traverse the edge connecting the other two vertices of the triangle. However, this is not possible because the corresponding $b$-edge lies on a non-trivial loop, and it must visit at least one of these vertices.
        \item We only consider the first column, as the remaining colums can be treated analogously. Observe that any non-trivial loop passing through edge $b_{1,1}$ must also traverse either $a_{1,1}$ or $c_{1,1}$. It cannot traverse both, since this would require including two additional edges of the triangle, resulting in a vertex of degree $3$ in the collection of loops, which is not allowed.
        
        Suppose, without loss of generality, that the non-trivial loop traverses $c_{1,1}$ (the other case can be treated in the same way). Then there are two possible ways for the loop to pass through the corresponding triangle: it may use either one edge or two edges of the triangle (see Figure \ref{fig:doubleedgeFisher}). In both cases, the non-trivial loop traversing the edge $b_{1,n}$ must also traverse the edge $c_{2,1}$. Repeating this argument inductively, we obtain that a collection of loops must traverse all edges $c_{i,1}$ for $1 \leq i \leq m$.
    \end{enumerate}
\end{proof}

\begin{figure}[h!]
    \centering
    \includegraphics[width=0.8\textwidth]{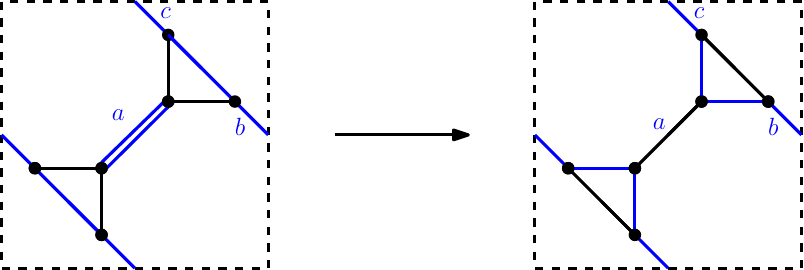}
    \caption{\label{fig:doubleedgeFisher} An operation that resolves double edges in the collections of loops.}
\end{figure}

Now, we express the roots of the corresponding marginal polynomials. 

\begin{thm} \label{rootsF}
    The polynomial $P_R(w)$ has $n$ real roots, and each root can be written as 
    \[r_j = (-1)^m \prod_{i = 1}^m \frac{a_{i,j}-a_{i,j}^{-1}}{c_{i,j} - c_{i,j}^{-1}},\]
    for $1 \leq j \leq n$.
\end{thm}

\begin{proof}
    By Lemma \ref{no double edges} (a), when we describe the polynomial $P_R(w)$, we know that all $b$-edges are included in the collections of loops. Therefore, we can independently look at each column of the graph $F_{m,n}$. Also, since we are describing the right boundary of the Newton polygon, this implies that all the $b$-edges are oriented from left to right. It is enough to describe the first column, and the remaining columns are done analogously. Similarly to the triangular lattice, besides the chosen Kasteleyn orientation on $F_{m,n}$, by \ref{char_poly_properties} (a) the number of loops in a given collection affects the sign of the contributions of the collections of loops. 

    To describe the coefficient of the monomial $w^k$ in $P_R(w)$ for $k \in \{-n, \dots, 0\}$, we need to describe all $\gamma \in \mathcal{L}(F_{m,n})$ contributing to this monomial with homology $(m,k)$. The number of non-trivial loops in $\gamma$ is $k$. By Lemma \ref{no double edges}(b), $\gamma$ contains no trivial loops other than double edges. Therefore, $\ell(\gamma) = d_1 + d_2$, where $d_1$ is the number of non-trivial loops in $\gamma$ and $d_2$ is the number of double edges in $\gamma$. Note that we can express $d_1 = \gcd(m,k)$. Therefore $(-1)^{d_1} = -1$ if $m$ is even, and $(-1)^k$ if $m$ is odd.
    By Lemma \ref{char_poly_properties} (a), the sign from the number of loops is equal to $(-1)^{|V(\Gamma)| - \ell(\gamma)} = (-1)^{6mn - d_1 - d_2} = (-1)^{d_1+d_2}$. Whenever we choose all $c$-edges in a given column, then we multiply the weights by $-(-1)^m$.

    By Lemma \ref{no double edges} (c), we can consider these two disjoint cases.

    \begin{enumerate}
        \item all $c$-edges are contained in the non-trivial loops.

        The contribution of the edge with weight $c_{1,1}$ is equal to  $-(-1)^m (-c_{1,1}w^{-1}) = (-1)^m c_{1,1}w^{-1}$, where the factor $-(-1)^m$ comes from the number of loops and $-c_{1,1}w^{-1}$ since the edge is oriented from left to right (the same as the corresponding $b$-edge). 
        
        Note that by putting a double edge on some $a$-edge, the number of loops increases by one. From this and the corresponding Kasteleyn signs, we get the following expression: 
        \begin{align*}
            (-1)^mc_{1,1}w^{-1}(1 - a_{1,1}^2)(-c_{2,1})(1 - a_{2,1}^2) \dots (-c_{m,1})(1 - a_{m,1}^2) =\\
            -(-1)^mc_{1,1} \dots c_{m,1} (a_{1,1}^2-1) \dots (a_{m,1}^2-1)w^{-1}.
        \end{align*}
        \item all $a$-edges are contained in the non-trivial loops.

        By similar analysis, there is a choice whether a $c$-edge is in a double edge or not contained in any loop. We get the following product:
        \[a_{1,1} \dots a_{m,1} (c_{1,1}^2-1) \dots (c_{m,1}^2-1).\]
        Note that by taking a double edge on $c_{1,1}$, the factors $w$ and $\frac{1}{w}$ cancel.

        So, we can now describe the first root $r_1$ from these two identities:

        \begin{align*}
            r_1 &= (-1)^m \frac{c_{1,1} \dots c_{m,1} (a_{1,1}^2-1) \dots (a_{m,1}^2-1)}{a_{1,1} \dots a_{m,1} (c_{1,1}^2-1) \dots (c_{m,1}^2-1)} \\
            &= (-1)^m \frac{(a_{1,1}-a_{1,1}^{-1}) \dots (a_{m,1}-a_{m,1}^{-1})}{ (c_{1,1}-c_{1,1}^{-1}) \dots (c_{m,1}-c_{m,1}^{-1})}.
        \end{align*}
    \end{enumerate}

    Following the same procedure for the other columns of the Fisher graph $F_{m,n}$, we obtain the remaining roots of the polynomial $P_R$.
\end{proof} 

As it was the case for the triangular lattice, these roots again have combinatorial meaning. Each root represents an alternating product along the zig-zag paths whose homology classes are primitive edge vectors of the corresponding sides of the Newton polygon. As opposed to the triangular lattice, here each weight $\wt$ is replaced by $\wt - \wt^{-1}$. 

\begin{thm} \label{real rooted marginals}
    Let $\Gamma$ be an isoradial bipartite torus graph, and let $S \subseteq V(\Gamma)$ be a subset of vertices of degree $3$. Then the marginal polynomials of $N(\Gamma_{\Delta_S})$ are real-rooted.
\end{thm}

\begin{proof}
    Without loss of generality, up to an $SL_2(\mathbb{Z})$, we can assume that $N(\Gamma)$ contains a vertical side. Since $\Gamma$ is isoradial bipartite, we can apply Corollary \ref{NPs and homology classes} to conclude that $N(\Gamma)$ is described by the homology classes of the zig-zag paths in $\Gamma$. Therefore, we can assume that $\Gamma$ has $r \geq 1$ zig-zag paths $\{Z_1, \dots, Z_r\}$ with homology class $(0,1)$.

    Using Lemma \ref{vertical zig-zag path}, the graph $\Gamma$ can be embedded in the torus without changing its characteristic polynomial, such that there is a zig-zag path, say $Z_r$, where the vertical boundary, denoted by $\gamma_y$, of the fundamental domain crosses all the edges of $Z_r$ and no other edges of $\Gamma$. As discussed in Lemma \ref{fundamental_domain matching}, each zig-zag path $Z_i$, $1 \leq i \leq r$ has exactly $k$ vertices, where $k$ describes the width of $N(\Gamma)$. We also fix the horizontal boundary $\gamma_x$ to be the curve that passes through every $Z_i$, for $1 \leq i \leq r$, exactly once. This is possible because every pair of zig-zag paths with homology $(0,1)$ is edge-disjoint. This choice of $\gamma_x$ and $\gamma_y$ does not affect $N(\Gamma)$. 

    Following the notation of Proposition \ref{rootsT}, we call the edges of $Z_i$ that follow immediately after the left (resp. right) turn $a$-edges (resp. $b$-edges). In order to describe the marginal polynomial attached to the vertical side of $N(\Gamma)$, we need to choose either all $a$-edges or all $b$-edges of the zig-zag path $Z_r$. In order to traverse the loops of total homology $k/2$ through $\Gamma$, we need to make the same choice on every zig-zag path $Z_i$. For each of them, we can choose either all $a$-edges or all $b$-edges, and by the isoradiality of $\Gamma$, these zig-zag paths are edge-disjoint. See the left part of Figure \ref{fig:Zi_in_Gamma_Gamma_DeltaS}. Therefore, the choice can be made independently for each zig-zag path $Z_i$, $1 \leq i \leq r$. 

    \begin{figure}[h!]
    \centering
    \includegraphics[width=0.65\textwidth]{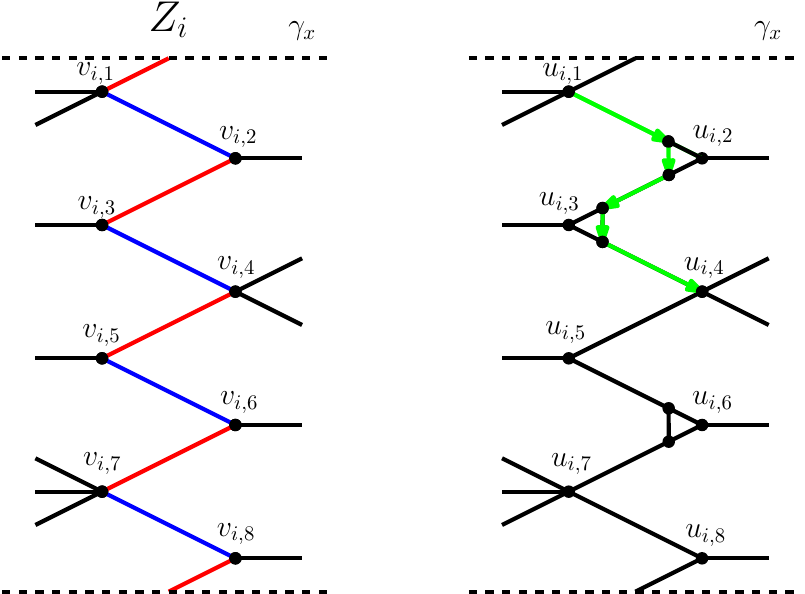}
    \caption{\label{fig:Zi_in_Gamma_Gamma_DeltaS} Left: the portion of the graph $\Gamma$ containing the zig-zag path $Z_i$ with homology $(0,1)$, its $a$-edges are depicted in blue, and $b$-edges are in red. The curve $\gamma_x$ is chosen such that it intersect every zig-zag path $Z_i$ exactly once. Right: the portion of the graph $\Gamma_{\Delta_S}$ obtained by replacing the vertices $v_{i,2}, v_{i,3}, v_{i,6}$ with triangles. The forbidden path from $u_{i,1}$ to $u_{i,4}$ is shown in green.}
    \end{figure}

    We now want to describe the marginal polynomials of $N(\Gamma_{\Delta_S})$. First, notice that by Theorem \ref{equal newton polygons delta}, we have $N(\Gamma_{\Delta_S}) = N(\Gamma)$. If $r = 1$, then the result follows because the marginal polynomial has degree $1$ and is always real-rooted. Therefore, we assume that $r \geq 2$.

    We denote by $V(Z_i) = \{v_{i,1}, \dots, v_{i,k}\}$ the vertex set of $Z_i$ for every $1 \leq i \leq r$. Take $v_{i,j} \in V(\Gamma)$, for $1 \leq j \leq k$. 
    If $v_{i,j} \notin S$, then we denote by $u_{i,j}$ the corresponding vertex in $V(\Gamma_{\Delta_S})$. If $v_{i,j} \in S$, then $u_{i,j}$ is the unique vertex in the triangle $\Delta_{v_{i,j}}$ that is not incident to any edge in $Z_i$ of the initial graph $\Gamma$. See the right part of Figure \ref{fig:Zi_in_Gamma_Gamma_DeltaS}.

    In order to describe the collections of loops in $\Gamma_{\Delta_S}$ whose first coordinate of the homology class is equal to $k/2$, instead of choosing $a$-edges and $b$-edges in $Z_i$, we need to consider the disjoint oriented paths going from $U^-_{i} = \{u_{i,1}, u_{i, 3}, \dots, u_{i, k-1}\}$ to $U^+_{i} = \{u_{i,2}, u_{i, 4}, \dots, u_{i, k}\}$. These paths are chosen such that the rest of the graph can be completed into a collection of loops that is realizable in $\mathcal{L}(\Gamma_{\Delta_S})$.

    It suffices to consider the $k/2$ oriented disjoint paths from $U^-_{i}$ to $U^+_{i}$ and double edges in order to cover all the vertices in $\Delta_{v_{i,j}}$. This covers the part of the graph $\Gamma_{\Delta_S}$ that is obtained from $Z_i$. By Lemma \ref{N(Gamma) < N(Gamma_Delta_v)}, using the original loops of $\Gamma$, we can find appropriate loops in $\Gamma_{\Delta_S}$. We denote by $P_{U^-_{i} \to U^+_{i}}(w)$ the characteristic polynomial describing these oriented paths and double edges.
    
    The only two ways in which these paths can be connected are $u_{i,1} \rightarrow u_{i,2}$, \dots, $u_{i,k-1} \rightarrow u_{i,k}$ or $u_{i,1} \rightarrow u_{i,k}$, \dots, $u_{i,k-1} \rightarrow u_{i,k-2}$. For example, if there is a path from $u_{i,1}$ to $u_{i,4}$, then there is no vertex in $U^-_{i}$ that can be connected to $u_{i,2}$ by a disjoint path. See the right part of Figure \ref{fig:Zi_in_Gamma_Gamma_DeltaS}. With similar reasoning, we can rule out the other options and notice that the only two possibilities are the ones we stated. When all the vertices of $Z_i$ are in $S$, this corresponds to the same argument we had in Lemma \ref{no double edges} c).

    Recall that we chose the curve $\gamma_x$ so that it intersects every $Z_i$ exactly once. By assuming that the added triangles are sufficiently small, $\gamma_x$ intersects only one of the edges in the corresponding portion of the graph $\Gamma_{\Delta_S}$. See Figure \ref{fig:Zi_in_Gamma_Gamma_DeltaS}. This implies that the degree of the polynomial $P_{U^-_{i} \to U^+_{i}}(w)$ is at most one for all $1 \leq i \leq r$. By Theorem \ref{equal newton polygons delta}, it must be equal to one.

    Using similar reasoning as in Lemma \ref{N(Gamma) < N(Gamma_Delta_v)}, we obtain that starting from a collection of loops in $\Gamma$, we can always obtain a collection of loops in $\Gamma_{\Delta_S}$ that uses the two possible collections of paths between $U^-_{i}$ and $U^+_{i}$. Let $P_{ver}(w)$ denote the marginal polynomial describing the vertical side of $N(\Gamma_{\Delta_S})$. We can write it as
    \[P_{ver}(w) = c P_{U^-_{1} \to U^+_{1}}(w) \dots P_{U^-_{r} \to U^+_{r}}(w).\]
    Here, $c$ is the constant that depends on the edge weights of the rest of the graph not contained in any of the paths between $U^-_{i}$ and $U^+_{i}$. This indeed does not depend on the variable $w$, because otherwise, the degree of the polynomial $P_{ver}(w)$ would be greater than $r$, contradicting the fact that $N(\Gamma_{\Delta_S}) = N(\Gamma)$. Since for each $1 \leq i \leq r$, the polynomial $P_{U^-_{i} \to U^+_{i}}(w)$ has a real root, the polynomial $P_{ver}(w)$ is also real-rooted. We can repeat this argument for every side of $N(\Gamma_{\Delta_S})$.
\end{proof}

\section{Local moves for non-bipartite graphs} \label{local_moves}

In this section, we are interested in studying local moves, transformations of graphs that preserve the dimer partition function up to scale. Goncharov and Kenyon~\cite{GoncharovKenyon} showed that any two minimal bipartite torus graphs are related by sequences of two types of local moves called spider moves and expanding/contracting degree two vertices. See Figure \ref{fig:local_bip_moves} and observe that both of these moves remain valid for non-bipartite graphs. 

Let us introduce a new move, which requires a graph to be non-bipartite. In Figure \ref{fig:localmove1}, we define the \emph{$\lambda$-move}, which involves four vertices such that its internal vertex has degree three, and the remaining vertices can be connected to the rest of the graph (which can also include an edge between the vertices $v_2$ and $v_3$). 

\begin{figure}[h!]
    \centering
    \includegraphics[width=0.65\textwidth]{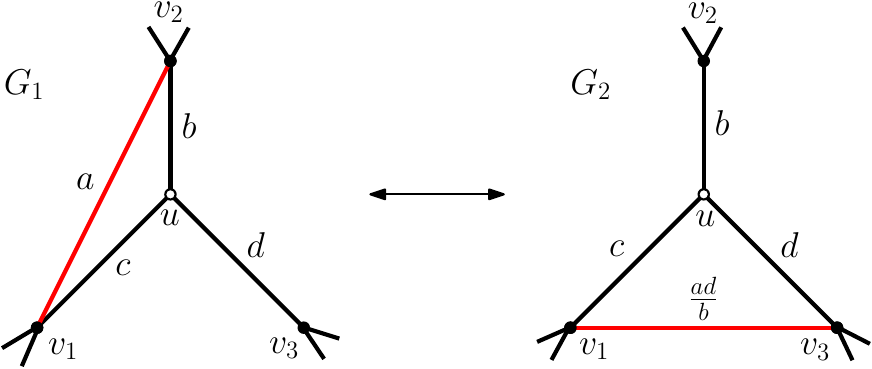}
    \caption{\label{fig:localmove1} $\lambda$-move, the central vertex $u$ has degree $3$ and the other vertices can be connected to the rest of the graph.}
\end{figure}

\begin{lem}
    A $\lambda$-move preserves the dimer partition function.
\end{lem}

\begin{proof}
    Denote by $G_1$ the initial graph and by $G_2$ the graph obtained by transforming $G_1$ according to the $\lambda$-move.

    First, notice that every dimer cover of $G_1$ that does not use the edge $v_1v_2$ can be identically mapped to the dimer cover of $G_2$ with the same weight. It remains to check the dimer covers that contain the edge $v_1v_2$. In that case, the internal vertex $u$ has to be matched with $v_3$. Thus, the total contribution of these two edges to the weight of the dimer cover is $ad$.
    
    On the other hand, we need to find a dimer cover in $G_2$ such that all the vertices $u, v_1, v_2, v_3$ are matched internally. The only way to match the vertex $v_2$ is to choose the edge $uv_2$, then $v_1$ and $v_3$ have to be matched. This gives the same weight $\frac{ad}{b} b = ad$ as for the graph $G_1$. Therefore, $G_1$ and $G_2$ have the same dimer partition function.
\end{proof}

The following move we treat is obtained by taking a standard spider move from Figure \ref{fig:local_bip_moves} and adding a red edge to both graphs to obtain a \emph{diagonal spider move}, as shown in Figure \ref{fig:nbspider}.

\begin{figure}[h!]
    \centering
    \includegraphics[width=0.7\textwidth]{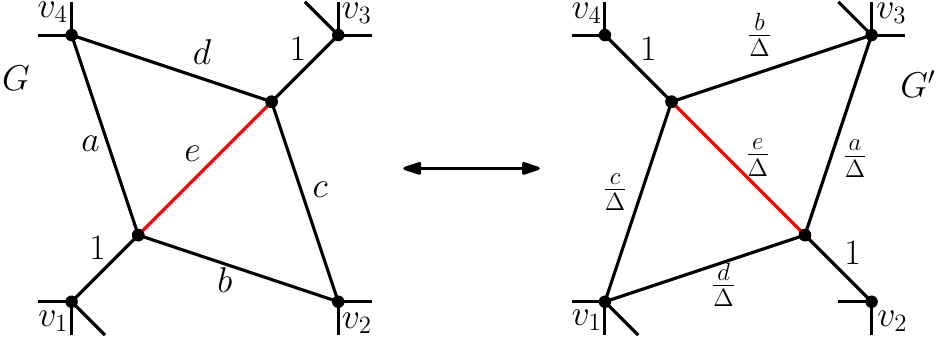}
    \caption{\label{fig:nbspider} A diagonal spider move, $\Delta := ac + bd$.}
\end{figure}

\begin{lem}
    Let $G'$ be the graph obtained by applying a diagonal spider move on a graph $G$. Then 
    \[Z(G') = \frac{1}{\Delta} Z(G),\]
    where $\Delta = ac + bd$.
\end{lem}

\begin{proof}
    If we remove the edge with weight $e$ from the graph, we obtain the standard spider move for which the desired identity has been proven. It remains to check whether the dimer covers that contain the edges with respective weights $e$ and $\tfrac e \Delta$ have the same total weights up to a factor of $\Delta$.
    
    If $e$ is in a dimer cover, then the remaining four vertices in $G$ of Figure \ref{fig:nbspider} are connected externally to the rest of the graph. The same holds on the right side for the edge with weight $\frac{e}{\Delta}$. Thus, the dimer partition function is still equal to the partition function of the graph on the left divided by $\Delta$.
\end{proof}

\subsection{Local moves preserving characteristic polynomials}

We can also consider the moves that preserve the characteristic polynomial up to scale. In the bipartite case, the coefficients of this polynomial describe the partition function of dimer covers with prescribed homology classes when superimposed on a reference dimer cover. For non-bipartite graphs, the characteristic polynomial describes a particular loop model. Recall Lemma \ref{char_poly_properties} (a). It would be natural to find a relation between these two definitions. 

\begin{prop}
    If two torus graphs $\Gamma_1$ and $\Gamma_2$ are related by a  contraction/expansion of a degree $2$ vertex, a spider move, a diagonal spider move, or a $\lambda$-move, then 
    \[P_1(z,w) = c_0 P_2(z,w).\]
    Where $P_1(z,w)$ (resp. $P_2(z,w)$) is the characteristic polynomial of $\Gamma_1$ (resp. $\Gamma_2$) and $c_0$ is a constant. 
\end{prop}

\begin{proof}
    We consider the four local moves separately.

    \begin{enumerate}[label=\Roman*.]
    \item \textbf{Contraction/expansion of a degree $2$ vertex.}
    
    Suppose that a vertex $v \in V(\Gamma_1)$ is replaced in $\Gamma_2$ by three vertices $v_1, v_2$ and $v_3$, such that $\deg(v_2) = 2$ and $\wt(v_1v_2) = \wt(v_2v_3) = 1$. See Figure \ref{fig:local_bip_moves} for a reminder of this move. Consider an oriented loop in $\Gamma_1$ that visits the vertex $v$. There are two cases.
    
    First, suppose that both neighbors of $v$ (in the case of a double edge, it is the same vertex) are adjacent to $v_1$ (resp. to $v_3$) in $\Gamma_2$. Then we keep the same oriented loop in $\Gamma_2$ and add a double edge on $v_2v_3$ (resp. $v_1v_2$), with contribution $1$.
    
    Second, suppose that one neighbor of $v$ is adjacent to $v_1$, and the other is adjacent to $v_3$. Then we can add the missing path $v_1 \rightarrow v_2 \rightarrow v_3$ (or $v_3 \rightarrow v_2 \rightarrow v_1$) to the oriented loop in $\Gamma_2$ without changing its weight.

    In both cases, the parity of the loop length is preserved. In particular, a trivial loop of even length in $\Gamma_1$ cannot become a trivial loop of odd length in $\Gamma_2$, which would not contribute to the characteristic polynomial by Lemma \ref{char_poly_properties} (b).
    
    \item \textbf{Spider move.}

    We call an \emph{internal path} (resp. double edge) a path (resp. double edge) that uses only the edges shown in Figure \ref{fig:nbspider}. We treat different cases depending on how external vertices $v_1, v_2, v_3, v_4$, shown in Figure \ref{fig:nbspider} (the weight $e = 0$), are connected to each other:

    \begin{enumerate} [label=\arabic*.]
        \item there is an internal path in $\Gamma_1$ between $v_1$ and $v_2$. Here, we distinguish three possibilities:
        \begin{enumerate} [(a)]
            \item there is a double edge adjacent to the vertex $v_3$ in $\Gamma_1$. So, the total contribution of this portion of the loop is equal to $b$. On the other hand, in $\Gamma_2$, we can have a path from $v_1$ to $v_2$ that passes through the vertex $v_3$, or a path from $v_1$ to $v_2$ and a double edge with weight $\frac{b^2}{\Delta^2}$ incident to $v_3$. This gives the contribution of $\frac{abc}{\Delta^3} + \frac{b^2d}{\Delta^3} = \frac{b}{\Delta^2}$.
            \item Symmetrically to the previous case, if there is an internal double edge adjacent to the vertex $v_4$ in $\Gamma_2$, whose weight is equal to $\frac{d}{\Delta}$. That corresponds to the path in $\Gamma_1$ from $v_1$ to $v_2$ that visits $v_4$ or a path from $v_1$ to $v_2$ and double edge on $v_4$ with weight $d^2$. The total weight is equal to $bd^2 + acd = d \Delta$.
            \item There is an internal path between $v_2$ and $v_3$, which gives a total contribution of $bc$. In $\Gamma_2$, we obtain a path from $v_1$ to $v_3$ and an internal double edge on $v_2$, with a total contribution of $\frac{bc}{\Delta^2}$.
            \item There is an internal path between $v_3$ and $v_4$. In $\Gamma_1$ the contribution is $bd$, while in $\Gamma_2$, it is $\frac{bd}{\Delta^2}$.
        \end{enumerate}
        \item There is an internal path between $v_1$ and $v_3$, not passing through $v_2$. Then it has to pass through the vertex $v_4$, and its contribution is $ad$. In $\Gamma_2$, this gives an internal double edge on $v_4$ and an internal path from $v_1$ to $v_3$, with contribution $\frac{ad}{\Delta^2}$.
        \item There is an internal double edge on $v_1$.
        \begin{enumerate} [(a)]
            \item There is a double edge on $v_3$. Its contribution is $1$. In $\Gamma_2$, this gives the following possibilities. A pair of double edges with total weight $\frac{a^2c^2}{\Delta^4}$, a pair of double edges with total weight $\frac{b^2d^2}{\Delta^4}$, or a trivial loop of length four with weight $\frac{abcd}{\Delta^4}$ with two possible orientations. So, the total contribution is $\frac{a^2c^2}{\Delta^4} + \frac{b^2d^2}{\Delta^4} + 2\frac{abcd}{\Delta^4} = \frac{(ac + bd)^2}{\Delta^4} = \frac{1}{\Delta^2}$.
            \item There is a double edge on $v_2$ (resp. $v_4$) with weight $c^2$ (resp. $d^2$). This gives two double edges on $v_1$ and $v_2$ (resp. $v_1$ and $v_4$) with weight $\frac{c^2}{\Delta}$ (resp. $\frac{d^2}{\Delta^2}$).
            \item There is a path from $v_2$ to $v_4$ in $\Gamma_1$ with weight $cd$. This corresponds to the path in $\Gamma_2$ from $v_2$ to $v_4$ and visits $v_1$ with total weight $\frac{cd}{\Delta^2}$.
        \end{enumerate}
        \item The vertex $v_1$ is not connected to any internal vertex. If there is a double edge adjacent to $v_3$, then this case is analogous to the cases $3(b)$ and $3(c)$. If $v_3$ is not internally connected, then it is analogous to the case $3(a)$. Finally, if there is a path between $v_3$ and $v_2$, or between $v_3$ and $v_4$, then it is analogous to the case $1(b)$
    \end{enumerate}

    \item \textbf{Diagonal spider move.}

    It is enough to consider the loops that pass through the edge with weight $e$ in Figure \ref{fig:nbspider}. Up to symmetry, there are three cases. 
    
    First, the edge $e$ is contained in the double edge of $\Gamma_1$. This corresponds to the double edge in $\Gamma_2$ with weight $\frac{e^2}{\Delta^2}$. 
    
    Second, a loop in $\Gamma_1$ contains an internal path from $v_1$ to $v_3$ that goes through the edges with weights $1$ and $e$, so the total weight is equal to $e$. In $\Gamma_2$, this corresponds to taking two loops, one with weight $\frac{ace}{\Delta^3}$ and the other with weight $\frac{bde}{\Delta^3}$. Their sum gives $\frac{e}{\Delta}$. 
    
    Third, up to symmetry, a loop in $\Gamma_1$ goes from $v_1$ to $v_2$ and uses the edge with weight $1, e, c$. In $\Gamma_2$, this is the path that uses the edges $\frac{c}{\Delta}, \frac{e}{\Delta}, 1$. We conclude that if $\Gamma_1$ and $\Gamma_2$ are related by a diagonal spider move, then \[P_1(z,w) = \frac{1}{\Delta^2} P_2(z,w).\]

    \item \textbf{$\lambda$-move.}

    It suffices to consider the loops containing the edge $a$ in $\Gamma_1$ in Figure \ref{fig:localmove1}. Up to symmetry, there are four possible cases:
    \begin{enumerate} [label=\arabic*.]
        \item a collection of loops in $\Gamma_1$ contains the double edges $a^2$ and $d^2$. This gives in $\Gamma_2$ the double edges $b^2$ and $\frac{a^2d^2}{b^2}$. Both contributions are equal;
        \item there is a loop in $\Gamma_1$ containing the path $v_1 \rightarrow v_2 \rightarrow u \rightarrow v_3$ with a total weight of $abd$. This corresponds to the path $v_1 \rightarrow v_3$ of weight $\frac{ad}{b}$ and the double edge $b^2$ in $\Gamma_2$. Both contributions are equal to $abd$;
        \item there is a loop in $\Gamma_1$ containing the path $v_1 \rightarrow v_2$ and the double edge $d^2$. This corresponds to the path $v_1 \rightarrow v_3 \rightarrow u \rightarrow v_2$ in $\Gamma_2$. Both contributions are equal to $ad^2$;
        \item there is a loop in $\Gamma_1$ containing the path $v_2 \rightarrow v_1 \rightarrow u \rightarrow v_3$ with total weight $acd$, which corresponds to the path $v_2 \rightarrow u \rightarrow v_1 \rightarrow v_3$ in $\Gamma_2$ whose weight is $\frac{bcad}{b} = acd$.
    \end{enumerate}
    \end{enumerate}
\end{proof}

\subsection{The non-planar local move}

We finish the section on local moves by going slightly beyond the scope of this paper. Namely, we discuss the \emph{cross move}, which can be seen in Figure \ref{fig:cross}. Depending on the connectivity of the boundary vertices $v_1,v_2,v_3,v_4$, the underlying graph on which the move is performed can be non-planar.
By taking the edge weight $f=0$, we recover the diagonal spider move that we discussed before. If we also take $e=0$, then we recover the standard spider move. We now show that this new move indeed preserves the dimer partition function up to a scale, which can be expressed.

\begin{figure}[h!]
    \centering
    \includegraphics[width=0.8\textwidth]{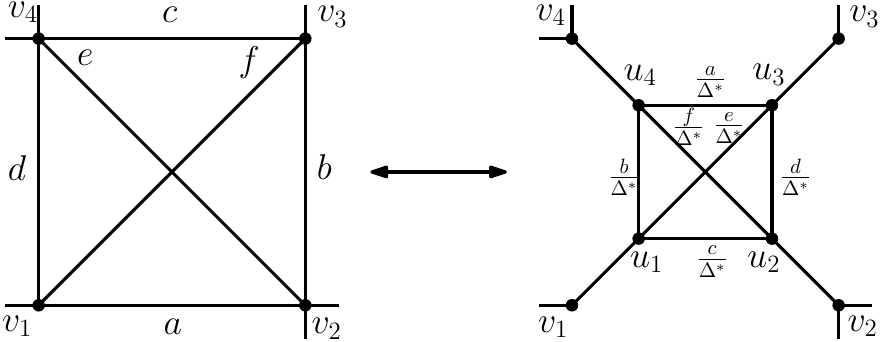}
    \caption{\label{fig:cross} A cross move, $\Delta^* := ac + bd + ef$.}
\end{figure}

\begin{lem}
    Let $G'$ be the graph obtained by applying a cross move on a graph $G$. Then 
    \[Z(G') = \frac{1}{\Delta^*} Z(G),\]
    where $\Delta^* = ac + bd + ef$.
\end{lem}

\begin{proof}
    If we take $e = f = 0$, then we obtain the usual spider move. So it suffices to check the cases when a dimer cover contains a subset of edges $e$ and $f$. Let $G$ be the graph on the left of Figure \ref{fig:cross}, and $G'$ be the graph on the right. First, consider the case when only one of the two edges of $G$ is in a dimer cover, say $e$. Then this forces the edges $v_2u_2$, $v_4u_4$, and $u_1u_3$ to be in the corresponding dimer cover in $G'$. Their total weight is equal to $\frac{e}{\Delta^*}$.
    
    If both $e$ and $f$ are in a dimer cover of $G$, this corresponds to the case when all the vertices $\{v_1, v_2, v_3, v_4\}$ are matched internally. So the total contribution to the partition function of $G$ is equal to $\Delta^*$. On the other hand, this means that all the edges $u_iv_i$, for $1 \leq i \leq 4$, whose total contribution is one, are in the corresponding dimer covers in $G'$. Thus, the result follows.
\end{proof}

One of the most studied examples of a planar bipartite graph in dimer theory is the Aztec diamond, introduced in~\cite{ElkKupLarPropp}. Propp~\cite{Propp03} applied spider moves to enumerate and generate perfect matchings of the Aztec diamond. Later on, Speyer~\cite{Speyer07} used the octahedron recurrence to generalize this result to a broader class of planar bipartite graphs whose perfect matchings can be enumerated.

Similarly to the correspondence between the octahedron recurrence and the spider move, we observe that, through the factor $\Delta^*$, the cube recurrence~\cite{CS04} is related to the cross move. This suggests the possibility of finding a non-planar analog of the Aztec diamond using this new move. Moreover, the obtained formula also resembles the Kuo condensation relation~\cite{Kuo}. See also Corollary 3.2 in~\cite{Sp16}. These connections suggest that the cross move may provide a natural framework for extending the combinatorial and algebraic structures underlying perfect matching enumeration beyond the classical planar bipartite setting.

\subsection*{Acknowledgements}

I would like to warmly thank my PhD advisors, Sanjay Ramassamy and Cédric Boutillier, for their guidance and support throughout this project, as well as for their valuable discussions, suggestions, and careful reading of the manuscript. I am also grateful to Béatrice de Tilière, Philippe Di Francesco, Terrence George, Elias Nohra, and Breki Pálsson for inspiring discussions.

\bibliographystyle{plainurl}
\bibliography{references}

\end{document}